\documentclass[letterpaper, 11pt,  reqno]{amsart}

\usepackage{amsmath,amssymb,amscd,amsthm,amsxtra, esint}

\allowdisplaybreaks[2]

\newtheorem{theorem}{Theorem} [section]

\newtheorem{lemma}[theorem]{Lemma}
\newtheorem{proposition}[theorem]{Proposition}
\newtheorem{remark}[theorem]{Remark}

\newtheorem{definition}[theorem]{Definition}
\newtheorem{corollary}[theorem]{Corollary}

\DeclareMathOperator*{\supp}{supp}

\newcommand{\I}{\hspace{0.5mm}\text{I}\hspace{0.5mm}}
\newcommand{\II}{\text{I \hspace{-2.8mm} I} }
\newcommand{\III}{\text{I \hspace{-2.9mm} I \hspace{-2.9mm} I}}

\newcommand{\noi}{\noindent}
\newcommand{\Z}{\mathbb{Z}}
\newcommand{\R}{\mathbb{R}}
\newcommand{\C}{\mathbb{C}}
\newcommand{\T}{\mathbb{T}}

\newcommand{\RR}{\mathbf{R}}

\newcommand{\K}{\mathbf{K}}

\newcommand{\F}{\mathcal{F}}

\def\norm#1{\|#1\|}

\newcommand{\al}{\alpha}
\newcommand{\be}{\beta}
\newcommand{\dl}{\delta}

\newcommand{\Dl}{\Delta}
\newcommand{\eps}{\varepsilon}

\newcommand{\ld}{\lambda}
\newcommand{\Ld}{\Lambda}
\newcommand{\s}{\sigma}

\newcommand{\ft}{\widehat}

\newcommand{\wt}{\widetilde}
\newcommand{\cj}{\overline}

\renewcommand{\l}{\ell}
\renewcommand{\o}{\omega}
\renewcommand{\O}{\Omega}

\newcommand{\les}{\lesssim}
\newcommand{\ges}{\gtrsim}

\newcommand{\jb}[1]
{\langle #1 \rangle}

\newcommand{\ind}{\mathbf 1}

\newcommand{\bn}{{\bf n}}

\newcommand{\bx}{{\bf x}}
\newcommand{\ba}{{\bf a}}

\newcommand{\bj}{{\bf j}}

\numberwithin{equation}{section}
\numberwithin{theorem}{section}

\begin{document}


\title[Strichartz estimates
on irrational tori]
{Strichartz estimates for Schr\"odinger equations
on irrational tori}

\author[Z.~Guo, T.~Oh, and Y.~Wang]{Zihua Guo, Tadahiro Oh, and Yuzhao Wang}

\address{
Zihua Guo\\
School of Mathematical Sciences\\
Peking University\\
Beijing 100871, China} \email{zihuaguo@math.pku.edu.cn}

\address{
Tadahiro Oh\\
School of Mathematics\\
The University of Edinburgh\\
and The Maxwell Institute for the Mathematical Sciences\\
James Clerk Maxwell Building\\
The King's Buildings\\
Mayfield Road\\
Edinburgh\\
EH9 3JZ, Scotland\\
and
Department of Mathematics\\
Princeton University\\
Fine Hall\\
Washington Rd.\\
Princeton, NJ 08544-1000, USA}

\email{hiro.oh@ed.ac.uk}

\address{
Yuzhao Wang\\
Department of Mathematics and Physics\\
North China Electric Power University\\
Beijing 102206, China}
\curraddr{Department of Mathematics and Statistics\\ 
Memorial University\\
St. John's, NL A1C 5S7, Canada}
\email{wangyuzhao2008@gmail.com}



\dedicatory{Dedicated to the memory of Professor Harold N. Shapiro (1922--2013)}

\begin{abstract}
In this paper, we prove new Strichartz estimates for linear Schr\"odinger
equations posed on $d$-dimensional irrational tori.
Then, we  use these estimates to prove  subcritical and
critical local well-posedness results
for nonlinear Schr\"odinger equations (NLS) on irrational tori.
\end{abstract}

\maketitle

\baselineskip = 15pt

\section{Introduction}\label{SEC:1}


\subsection{Background}
The Cauchy problem of the nonlinear Schr\"odinger equation (NLS):
\begin{equation}
\begin{cases}\label{NLS0}
i \partial_t u - \Delta u = \pm |u|^{p-1}u \\
u\big|_{t = 0} = u_0 \in H^s(M),
\end{cases}
\qquad (\bx, t) \in M \times \R
\end{equation}

\noi has been studied extensively in different settings (for
example, $M = \R^d$, $\T^d$, and certain classes of manifolds) over
recent years \cite{GV2, Tsutsumi, Kato, CazW2, Bo2, Bo1,BGT1, BGT2,
KochT, HTT11, HTT2, Herr, W}. See also the following monographs
\cite{SULEM, Caz, TAO} for  references therein. In the study of NLS
(1.1), Strichartz estimates of the following type have played a
fundamental role
\begin{equation}\label{IntroStr}
\| e^{- it\Dl} f\|_{L^q_t L^r_\bx (\R\times M)} \les
\|f\|_{H^s_\bx(M)},
\end{equation}

\noi
where
$\|f\|_{L_t^qL_\bx^r}=\big\|\norm{f(\bx, t)}_{L^r_\bx}\|_{L_t^q}$.\footnote{
We use $A\lesssim B$ to denote an estimate of the form $A\le CB$ for some $C>0$. 
Similarly, we use $A\sim B$ to denote $A\lesssim B$ and
$B\lesssim A$.}
In particular, when $M = \R^d$, \eqref{IntroStr} is known to hold
with $s = 0$, namely
\begin{equation}\label{IntroStr2}
\| e^{- it\Dl} f\|_{L^q_t L^r_\bx (\R\times \R^d)} \les
\|f\|_{L^2_\bx(\R^d)},
\end{equation}

\noi
if and
only if
 $(q, r)$ satisfies  $\frac{2}{q} + \frac{d}{r} = \frac{d}{2}$ with $2\leq q, r \leq \infty$
and $(q, r, d) \ne (2, \infty, 2)$. 
See \cite{Strichartz, Yajima, GV, KeelTao}.
This was first obtained
for the case $q=r$ by Strichartz \cite{Strichartz} via
the Fourier restriction method. 
It was then generalized by a
combination of the duality argument and the following dispersive estimate:
\begin{equation}\label{eq:dis}
\| e^{- it\Dl} f\|_{L^\infty_\bx (\R^d)} \lesssim
|t|^{-\frac d2}\|f\|_{L^1_\bx(\R^d)}.
\end{equation}

\noi
The endpoint case $(q,r)=(2,\frac{2d}{d-2})$, $d \ne 2$,  was then proven in
\cite{KeelTao}.

Now, consider
$f\in L^2(\R^d)$ with $\supp \ft f \subset [-N, N]^d$.
Then, as an immediate corollary to \eqref{IntroStr2}, we have the following
 Strichartz estimate on $\R^d$:
\begin{align}
\| e^{-it \Delta} f\|_{L^p_{t,\bx}(\R \times \R^d)} \lesssim  N^{\frac{d}{2} - \frac{d+2}p}\|f\|_{L^2(\R^d)},
\label{IntroStr3}
\end{align}

\noi
 for $\frac{2(d+2)}{d} \leq p \leq \infty$
on $\R^d$.
Indeed, 
%
on the one hand, 
 the Strichartz estimate  \eqref{IntroStr2} with $q = r = \frac{2(d+2)}{d}$ gives
\begin{equation}\label{P3}
\|e^{-it \Dl}f\|_{L^\frac{2(d+2)}{d}_{t, \bx}(\R\times \R^d)}
\les \|f\|_{L^2(\R^d)}.
\end{equation}

\noi
On the other hand, by Bernstein's inequality \cite[Chapter 5]{Wo}, we have 
\begin{equation}\label{P4}
\|e^{-it \Dl}f\|_{L^\infty_{t, \bx}(\R\times \R^d)}
\les N^\frac{d}{2} \|f\|_{L^2(\R^d)}, 
\end{equation}

\noi
for all $f\in L^2(\R^d)$ with $\supp \ft f \subset [-N, N]^d$.
By interpolating \eqref{P3} and \eqref{P4}, 
we see that 
the estimate \eqref{IntroStr3}  holds for $\frac{2(d+2)}{d} \leq p \leq \infty$
and is sharp in view of sharpness of \eqref{P3} and \eqref{P4}.
Note that the estimate \eqref{IntroStr3} is {\it scaling-invariant}
in the following sense.
Consider the linear Schr\"odinger equation:
\begin{equation}
\begin{cases}
i \partial_t u - \Delta u = 0\\
u\big|_{t = 0} = f.
\end{cases}
\label{linSchro}
\end{equation}

\noi
The solution $u$ to \eqref{linSchro} is given by $u (\bx, t) := e^{-it\Dl} f(\bx)$.
Then, the rescaled function $u^\ld(\bx, t) : = u(\ld \bx, \ld^2 t)$, $\ld > 0$, 
is also a solution to \eqref{linSchro} but with the rescaled initial condition
$f^\ld(\bx): = f(\ld \bx)$.
Noting that $\supp \ft{f^\ld} = \ld \cdot \supp \ft f$, 
it is easy to see that 
the power of $N$ in \eqref{IntroStr3}
is the only power that is consistent with this scaling.
We point out that the inequalities \eqref{IntroStr2}, \eqref{P3}, and \eqref{P4}
are also scaling-invariant with respect to this scaling
associated to the linear Schr\"odinger equation
and that the scaling-invariance shows sharpness of these estimates.

When $M$ is a compact manifold, the Strichartz estimate \eqref{IntroStr}
becomes much more difficult and much less is known. 
This is partially due to the fact that we do not
have the dispersive estimate \eqref{eq:dis} on a compact manifold.
Moreover, \eqref{IntroStr}
requires deep understanding of the eigenvalues and the eigenfunctions of the Laplacian. 
In the following, 
we  focus on the case when $M$ is a standard flat torus $\T^d = (\R/ \Z)^d$,
corresponding to the usual periodic boundary condition.
Moreover, we restrict our attention to the diagonal case, i.e.~$q = r$.
Then,  one would like to establish the following scaling-invariant\footnote{Obviously, the scaling associated to the linear Schr\"odinger equation discussed above for $\R^d$ does not quite make sense on $\T^d$.
We nonetheless call the estimate \eqref{eq:stritori} scaling-invariant.} Strichartz estimate: 
\begin{align}\label{eq:stritori}
\| e^{-it \Delta} f\|_{L^p_{t,\bx}(I \times \T^d)} \lesssim
N^{\frac{d}{2} - \frac{d+2}p}\|f\|_{L^2(\T^d)},
\end{align}
for all  $f \in L^2(\T^d)$ with $\supp \ft f \subset [-N, N]^d$,
where $I$ is a compact interval.
Note that,
 in the compact setting,  an estimate of the form \eqref{eq:stritori} does not 
hold with $I = \R$, unless $p = \infty$.

By drawing an analogy to the Euclidean case $M = \R^d$, 
one may hope to have \eqref{eq:stritori}
for  $p \geq \frac{2(d+2)}{d} $.
By combining the tools from number theory, such as a divisor counting argument 
and the Hardy-Littlewood circle method,
and the Tomas-Stein restriction method from harmonic analysis, 
Bourgain \cite{Bo2} proved 
\eqref{eq:stritori} for certain ranges of $p$:
(i) $ p > \frac{2(d+2)}{d}$ when $ d = 1, 2$, (ii) $p > 4$ when $ d=  3$, and
(iii)  $p > \frac{2(d+4)}d$ for higher dimensions $d \geq 4$. 
It is worthwhile to note
that, when $d = 1, 2$, \eqref{eq:stritori} is known to fail at the
endpoint $p = \frac{2(d+2)}{d}$. See \cite{Bo2, TT}. 
Namely, the situation on $\T^d$ is strictly worse than the Euclidean setting.
Indeed,  Bourgain \cite{Bo2,Bo4} conjectured that
\begin{align}\label{P01}
\| e^{-it \Delta} f\|_{L^p_{t,\bx}(I \times \T^d )} \le  K_{p,
N}\|f\|_{L^2(\T^d)},
\end{align}

\noi
for all   $f \in L^2(\T^d)$ with $\supp \ft f \subset [-N, N]^d$,
where $K_{p,N}$ satisfies 
\begin{equation}
\begin{cases}
K_{p,N} < c_p,     & \text{if } p < \frac{2(d+2)}{d},  \\
K_{p,N} \ll N^\eps,   &\text{if }  p=\frac{2(d+2)}{d},  \\
K_{p,N} < c_p N^{\frac{d}2 - \frac{d+2}p}, \quad   &
\text{if } p>\frac{2(d+2)}{d},
\end{cases}\label{conj}
\end{equation}

\noi
for any small $\eps > 0$.
Here,  $A(N) \ll B(N)$ means that 
$ \lim_{N\to
\infty} \frac{A(N)}{B(N)} = 0$. 
More recently,
using multilinear restriction theory after \cite{BCT, BG}, Bourgain
\cite{Bo5}   improved the result (iii) for $d \geq 4$ and showed
that \eqref{eq:stritori} holds for $p > \frac{2(d+3)}d$. 
The general conjecture \eqref{conj}, however, remains open up to date.
In \cite{Bo2, Bo1,HTT11, W}, these Strichartz estimates were then applied
to prove well-posedness results of NLS \eqref{NLS0} on $\T^d$.
See Subsection \ref{SUBSEC:1.3} for more on the well-posedness issue of \eqref{NLS0}.

Let us conclude this subsection by stating the result 
by Herr \cite{Herr}.
He considered  the quintic NLS on a three-dimensional Zoll manifold $M$,
i.e.~a compact Riemannian manifold such that all geodesics are simple
and closed with a  {\it common minimal period}.
One simplest example is 
the three dimensional sphere $\mathbb{S}^3$.
By establishing the Strichartz estimate \eqref{eq:stritori}
on $M$ (instead of $\T^d$) with $p >4$, 
he proved local well-posedness of the quintic NLS 
on a three-dimensional Zoll manifold $M$
in the energy space $H^1(M)$.
As mentioned above, 
all geodesics on a Zoll manifold have a  common minimal period.
Hence, it is natural to ask 
if  a Strichartz estimate of the form \eqref{eq:stritori} holds
on a manifold,  
where there is no common minimal period for geodesics.
This leads us to the study of Strichartz estimates
on 
an {\it irrational torus}
$\T^d_{\pmb{\al}}$,
since it is one of the simplest examples of manifolds
with no common minimal period for geodesics.

\subsection{Strichartz estimate on irrational tori}
In the remaining part of this paper, we focus on the case 
when $M$ is an irrational torus $\T^d_{\pmb{\al}}$:
\begin{equation}
 M = \T^d_{\pmb{\al}}:= \prod_{j = 1}^d \R/ ( \al_j \Z), \quad \al_j > 0,\  j = 1, \dots,
 d.
\label{Torus}
 \end{equation}

\noi 
As the name suggests, we are mainly interested
in the case when at least one $\al_j$ is irrational.
More generally, 
we are interested in the case
when at least one $\al_{j}$ is ``rationally independent''
of the remaining ones, i.e.~there exists $\al_{j}$ that can not be written as a linear 
combination of the other $\al_k$'s with rational coefficients.

First consider the case when all $\al_j$'s are rational.
Namely, $M = \T^d_{\pmb{\al}}$ is a ``rational'' torus.
In this case, the problem can be reduced to that on the standard torus $\T^d$
by a simple geometric consideration.
By writing  $\al_j = \frac{k_j}{m_j}$ for some $k_j, m_j \in \mathbb{N}$, 
let $k$ be the least common multiple of $k_j$'s.
The basic idea is to view 
the scaled standard torus  $\wt M: = k  \T^d = \big( \R/(k \Z) \big)^d$
as a disjoint union of  parallel translates of the original rational torus $M = \T^d_{\pmb{\al}}$
with $\al_j^{-1} k$ copies in the $x_j$-direction.

Now, consider the linear Schr\"odinger equation \eqref{linSchro}
on $M = \T^d_{\pmb{\al}}$.
By periodic extension, 
we can view this problem on the scaled standard torus $\wt M = k  \T^d $.
Given an initial condition $f$ and the solution $u(t) = e^{-it \Dl} f$ on $M$, 
let $\wt f$ and $\wt u$ denote their periodic extensions
on $\wt M$, respectively.
By uniqueness of solutions to the linear Schr\"odinger equation, 
we see that $\wt u(t) = e^{-it \Dl} \wt f$ on $\wt M$.
Clearly, the Strichartz estimates on the standard torus $\T^d$
also hold on the scaled standard torus $\wt M = k \T^d $,
where the implicit constants further depend on $k$.
With $\pmb{\al} = (\al_1, \dots, \al_d)$, we have 
\[\|\wt f \|_{L^2_\bx(\wt M)} = C(\pmb{\al}) \|f \|_{L^2_\bx(M)}
\quad \text{and}\quad 
\|\wt u(t) \|_{L^p_\bx(\wt M)} = C(\pmb{\al}, p) \|u(t) \|_{L^p_\bx(M)}.
\]

\noi
Moreover, 
letting 
$\ft f$ and  $\wt{\mathcal{F}} [\wt f]$ 
denote the  Fourier coefficients of $f$ on $\T^d_{\pmb{\al}}$
and $\wt f$ on $k \T^d$, respectively, we have  
\begin{align*}
 \ft f (\bn )
&  = \frac{1}{|\T^d_{\pmb{\al}}|}
\int_{\T^d_{\pmb{\al}}} f(\bx) e^{-2\pi i \sum_{j = 1}^d n_j \frac{  x_j}{\al_j}} d\bx\\
& = \frac{1}{|k \T^d|}
\int_{k \T^d} \wt f(\bx) e^{-2\pi i \sum_{j = 1}^d \frac{k n_j}{\al_j} \frac{x_j}{k}} d\bx
= \wt{\mathcal{F}}[\wt f]
\big(\tfrac{k}{\al_1}n_1, \dots, \tfrac{k}{\al_d}n_d\big),
\end{align*}

\noi
where $\bn = (n_1, \dots, n_d) \in \Z^d$
and $|\,\cdot\,|$ denotes the Lebesgue measure of a set. 
Namely, we have
\[ \supp \wt{\mathcal{F}}[\wt f] 
= \frac{k}{\pmb{\al}}\cdot \supp \ft f
:=
\Big\{\big(\tfrac{k}{\al_1}n_1, \dots, \tfrac{k}{\al_d}n_d\big) \in \Z^d:\,
\bn \in \supp \ft f \, \Big\}.
\]

\noi
Therefore, 
we see that the Strichartz estimates of the form \eqref{P01}
on the standard torus $\T^d$
also hold  
on our rational torus $M = \T^d_{\pmb{\al}}$,
where the implicit constants further depends
on $\pmb{\al} = (\al_1, \dots, \al_d)$.
When there is no  $\al_{j}$ that is rationally independent
of the remaining ones, 
we can use  spatial and temporal dilations 
to reduce the situation to the case of a rational torus above.
Therefore, in the following, we assume that 
 at least one $\al_{j}$ is rationally independent
of the remaining $\al_k$'s.

Before proceeding further, let us  change the spatial domain 
 $M = \T^d_{\pmb{\al}}$ to the standard torus $\T^d$
 at the expense of modifying the Laplacian.
By a change of spatial variables ($x_j \mapsto \al_j x_j$),
we see that 
 \eqref{NLS0} is equivalent to the following NLS on the
usual torus $\T^d = (\R / \Z)^d$:
\begin{equation}
\begin{cases}\label{eq:nls}
i \partial_t u - \Delta u = \pm |u|^{p-1}u\\
u\big|_{t = 0} = u_0\in H^s(\T^d),
\end{cases}
 \quad (\bx, t) \in \T^d\times  \R,
\end{equation}

\noi where the Laplace operator $\Delta$ is now defined by
\begin{equation}\label{eq:Q0}
\widehat{\Delta f}(\bn) =  - 4 \pi^2 Q(\bn)\widehat f(\bn),
\end{equation}

\noi with $\bn = (n_1,\dots, n_d) \in \Z^d$ and
\begin{align}\label{eq:Q}
Q(\bn) = \theta_1 n_1^2 + \cdots + \theta_d n_d^2, \quad
\tfrac{1}{C}\leq\theta_j: = \tfrac{1}{\al_j^2}\leq C,  \  j=1,\cdots,d.
\end{align}

\noi
We point out that some estimates in the following depend
on $C$ in \eqref{eq:Q} but not on the specific arithmetic nature of $\theta_j$'s.

Our main interest is to discuss well-posedness of  the Cauchy problem \eqref{eq:nls}
by first studying relevant Strichartz estimates in this setting.
As compared to the problem on the standard torus $\T^d$, 
i.e.~with $Q(\bn) = |\bn|^2 = \sum_{j = 1}^d n_j^2$, 
it is a lot harder to study Strichartz estimates on irrational tori. 
The main reason for this difficulty
is that the number theoretic tools such as 
a divisor counting argument and the Hardy-Littlewood circle method 
do not work well in this setting.

Previously, Bourgain \cite{Bo4} and Catoire-Wang \cite{CW}
 studied the Cauchy problem \eqref{eq:nls} on irrational tori
and proved some local well-posedness results in subcritical Sobolev
spaces. See Theorem \ref{THM:3} below. In the following, we
investigate new Strichartz estimates on irrational tori and use them
to prove well-posedness results of the Cauchy problem \eqref{eq:nls}
in both subcritical and critical Sobolev spaces. In the rest of the
paper, we assume that the Laplacian $\Delta$ is defined by \eqref{eq:Q0}, unless
stated otherwise, and define
the linear Schr\"odinger evolution by\footnote{Strictly speaking,
 there is an extra factor of $2\pi$ in front of $Q(\bn)$ in \eqref{eq:Q}.
 However, such a factor can be eliminated by
 time dilation and thus, for simplicity of notations,  we drop it  in the
 following.}\begin{align}\label{lin-evo}
e^{-it\Delta} f (\bx) = \sum_{\bn\in \Z^d} \ft f(\bn) e^{2\pi i (\bn \cdot
\bx + Q(\bn) t)},
\end{align}

\noi
where $Q(\bn)$ is as in \eqref{eq:Q}.
We first summarize the known Strichartz estimates.
In the following,  $I$ denotes a compact interval in $\R$.

\begin{theorem}\label{THM:1}
The Strichartz estimate on a irrational torus is known to hold
\begin{align}
\| e^{-it \Delta} f\|_{L^p_{t,\bx}(I \times \T^d )} \lesssim  K_{p, N}\|f\|_{L^2(\T^d)},
\label{P0}
\end{align}

\noi
for all   $f \in L^2(\T^d)$ with $\supp \ft f \subset [-N, N]^d$
in the following cases:
\begin{itemize}
\item[(i)]  $d= 2$ \cite{CW}\footnote{
After the completion of this manuscript,
we learned that this result in \cite{CW} was recently improved
to $K_{4, N} = N^{\frac{131}{832}+}$
by Demirbas \cite{Demirbas}.
While the proof in \cite{CW} is based on Jarn\'ik's argument \cite{Jarnik},
the proof in \cite{Demirbas} is based on Huxley's counting estimate \cite{Huxley}.
More recently, this result in \cite{Demirbas} was improved 
to $K_{4, N} = N^{\frac{1}{8}+}$
by Demeter \cite{Demeter}.
See the footnote in Theorem \ref{THM:2} (ii).}
\textup{:} $K_{4, N} = N^\frac{1}{6}$,
\item[(ii)] $d=3$ \cite{Bo4}\textup{:}  $K_{4, N} = N^{\frac{1}{3}+\eps}$,
\item[(iii)] $d\ge3$ \cite{CW}\textup{:}
  $K_{4, N} = N^{\frac{d}{4} - \frac d{2(d+1)} +\eps}$ when $d$ is odd,
  and
  $K_{4, N} =  N^{\frac{d}{4} - \frac 1{2} +\eps}$ when $d$ is even,
\item[(iv)] $d \geq 2$ \cite{Bo5}\textup{:} $K_{p, N} = N^\eps$ for $p = \frac{2(d+1)}{d}$,
\end{itemize}

\noi
for any small $\eps >0$.
\end{theorem}

\noi Note that the implicit constants in \eqref{P0}
 depend on  $C$ in \eqref{eq:Q} and the length of the local-in-time interval $I$.
The same comment applies to all the estimates in the remaining of the paper
and we do not mention this dependence explicitly in the following.
In \cite{Bo4}, Bourgain also proved
\begin{equation}
\| e^{-it \Delta} f\|_{L^p_{t}L^4_\bx(I\times \T^3)} \lesssim N^{\frac{3}{4} - \frac{2}p}\|f\|_{L^2(\T^3)}
\label{P0a}
\end{equation}
for $p>\frac {16}3$.

In this paper, we partially improve the known results in Theorem
\ref{THM:1}, and obtain some critical Strichartz estimates when $p$
is large. We state our main result on the Strichartz estimates on
irrational tori.

\begin{theorem}\label{THM:2}
\textup{(i)} The following scaling-invariant Strichartz estimate holds on an irrational torus:
\begin{align}
\| e^{-it \Delta} f\|_{L^p_{t,\bx}(I \times \T^d)} \lesssim  N^{\frac{d}{2} - \frac{d+2}p}\|f\|_{L^2(\T^d)},
\label{P1}
\end{align}

\noi
for all $f \in L^2(\T^d)$ with $ \supp \ft f \subset [-N,N]^d$,
provided that  $d$ and $p$ are in the following ranges:
\begin{itemize}
\item[(i.a)] $d=2$\textup{:} $p>\frac{20}{3}$,
\item[(i.b)]$d=3$ 
\textup{:} $p>\frac {16}3$,
\item[(i.c)] $d=4$\textup{:} $p>4$,
\item[(i.d)] $d\ge 5$\textup{:} $p\ge 4$.
\end{itemize}

\noi
\textup{(ii)}
Let $\eps > 0$.
Then, the Strichartz estimate with an $\eps$-loss of regularity holds on an irrational torus:\footnote{In a very recent preprint, Demeter \cite{Demeter} proved the
Strichartz estimates  \eqref{P2} with an $\eps$-loss on the standard torus $\T^d$  for $p \geq \frac{2(d+3)}{d}$.
His argument is based on  incidence geometry,  without  any number theory.
As a result,  the same result holds for irrational tori
and hence improves our result in Theorem \ref{THM:2} (ii) in a significant manner.
This also improves the values of $s_0$ in 
some {\it subcritical} local well-posedness results below (Theorem \ref{THM:3} (i.a), (i.b), (ii.a),  and (ii.b)).
Note that the result in \cite{Demeter} comes with an $\eps$-loss and thus it does not improve
the scaling-invariant Strichartz estimate \eqref{P1} in Theorem \ref{THM:2} (i)
and {\it critical} local well-posedness results in Theorems \ref{THM:4} and \ref{THM:5}.}
\begin{align}
\| e^{-it \Delta} f\|_{L^p_{t,\bx}(I \times \T^d )} \lesssim  N^{\frac{d}{2} - \frac{d+2}p+\eps}\|f\|_{L^2(\T^d)},
\label{P2}
\end{align}

\noi
for all $f \in L^2(\T^d)$ with $ \supp \ft f \subset [-N,N]^d$,
provided that  $d$ and $p$ are in the following ranges:
\begin{itemize}
\item[(ii.a)] $d=2$\textup{:} $p\geq \frac{20}{3}$,
\item[(ii.b)]$d=3$\textup{:} $p= \frac {16}3$,
\item[(ii.c)] $d=4$\textup{:} $p= 4$.
\end{itemize}

\end{theorem}

\noi
When $d\geq 3$,
we follow a relatively simple argument  after Bourgain \cite{Bo4}
and prove Theorem \ref{THM:2} in Subsection \ref{SUBSEC:d3}.
When $d = 2$, this argument proves \eqref{P2} only for $p \geq 8$.
In Subsection \ref{SUBSEC:d2},
we present a duality argument to prove \eqref{P1} (when $d = 2$) for $p > 12$.
In Section \ref{SEC:level},
we establish  certain level set estimates
and provide a full proof of  Theorem \ref{THM:2} when $d = 2$.


By interpolating with Theorems \ref{THM:1} and \ref{THM:2},
we can obtain Strichartz estimates for other values of $p$.
The following corollary shows a summary of  known Strichartz estimates on irrational tori
at this point.

\begin{corollary}\label{COR:1}
Let $\eps > 0$.
Then, the following Strichartz estimates hold
for all
 $f \in L^2(\T^d)$ with $\supp \ft f \subset [-N,N]^d$:

\medskip

\noi
\textup{(i)} $d = 2$\textup{:}
\begin{align}\label{co-d2}
\| e^{-it \Delta} f\|_{L^p_{t,\bx}(I\times\T^2)}
\lesssim
\begin{cases}
\vphantom{\Big|}
N^{\eps}\|f\|_{L^2(\T^2)}, & \text{for } 2 < p \leq 3, \\
\vphantom{\Big|} N^{\frac 23 - \frac 2p+\eps}\|f\|_{L^2(\T^2)}, & \text{for } 3 < p < 4, \\
\vphantom{\Big|} N^{\frac16}\|f\|_{L^2(\T^2)}, & \text{for } p = 4, \\
\vphantom{\Big|}N^{\frac34 - \frac 7{3p}+\eps}\|f\|_{L^2(\T^2)},&  \text{for } 4<p\leq\frac{20}{3}, \\
\vphantom{\Big|} N^{1 - \frac{4}p }\|f\|_{L^2(\T^2)},& \text{for } p> \frac{20}{3}.
\end{cases}
\end{align}


\noi
\textup{(ii)} $d = 3$\textup{:}
\begin{align}\label{co-d3}
\| e^{-it \Delta} f\|_{L^p_{t,\bx}(I\times\T^3)}
\lesssim
\begin{cases}
\vphantom{\Big|}
N^{\eps}\|f\|_{L^2(\T^3)}, & \text{for } 2 < p \leq \frac 83, \\
\vphantom{\Big|}
N^{1 - \frac{8}{3p}+\eps}\|f\|_{L^2(\T^3)}, & \text{for } \frac 83 < p \leq 4, \\
\vphantom{\Big|}
N^{\frac54 - \frac{11}{3p}+\eps}\|f\|_{L^2(\T^3)},&  \text{for } 4< p\leq \frac{16}{3}, \\
\vphantom{\Big|}
N^{\frac 32  - \frac{5}p }\|f\|_{L^2(\T^3)},& \text{for } p> \frac{16}{3}.
\end{cases}
\end{align}

\noi
\textup{(iii)} $d = 4$\textup{:}
\begin{align}\label{co-d4}
\| e^{-it \Delta} f\|_{L^p_{t,\bx}(I\times\T^4)}
\lesssim
\begin{cases}
\vphantom{\Big|}
N^{\eps}\|f\|_{L^2(\T^4)}, & \text{for } 2 < p \leq \frac 52, \\
\vphantom{\Big|}
N^{\frac 43 - \frac{10}{3p}+\eps}\|f\|_{L^2(\T^4)}, & \text{for } \frac 52 < p \leq 4, \\
\vphantom{\Big|}
N^{ 2  - \frac{6}p }\|f\|_{L^2(\T^4)},& \text{for } p> 4.
\end{cases}
\end{align}

\noi
\textup{(iv)} $d \geq 5$\textup{:}
\begin{align}\label{co-d5}
\| e^{-it \Delta} f\|_{L^p_{t,\bx}(I\times\T^d)}
\lesssim
\begin{cases}
\vphantom{\Big|}
N^{\eps}\|f\|_{L^2(\T^d)}, & \text{for } 2 < p \leq \frac{2(d+1)}{d} , \\
\vphantom{\Big|}
N^{(\frac d4 - \frac 12)
(\frac{2d}{d-1} - \frac{4(d+1)}{p(d-1)})+\eps}\|f\|_{L^2(\T^d)}, & \text{for } \frac{2(d+1)}{d}  < p < 4, \\
\vphantom{\Big|}
N^{ \frac{d}{2}  - \frac{d+2}p }\|f\|_{L^2(\T^d)},& \text{for } p\geq 4.
\end{cases}
\end{align}

\end{corollary}

\subsection{Local well-posedness results of NLS on irrational tori}
\label{SUBSEC:1.3}

In the following,  we apply these Strichartz estimates in Corollary \ref{COR:1} to the
Cauchy problem of NLS on an irrational torus:
\begin{equation}
\begin{cases}\label{eq:nls1}
i \partial_t u - \Delta u =\pm |u|^{2k}u \\
u\big|_{t = 0} = u_0 \in H^s(\T^d),
\end{cases}
\quad (\bx, t) \in \T^d\times  \R,
\end{equation}
where $k\in \mathbb{N}$ is a positive integer
and  the Laplacian $\Dl$ is defined by \eqref{eq:Q0}.
%
%
%
%
%
%
%
%
%
%
First, recall the following notion.
When $M = \R^d$, the Cauchy problem \eqref{NLS0} enjoys the dilation symmetry.
Namely, if $u$ is a solution to \eqref{NLS0} with respect to
an initial condition  $u_0$,
then the rescaled function $u_\ld(\bx, t) := \ld^{\frac{2}{p-1}} u(\ld \bx, \ld^2 t)$
is also a solution to \eqref{NLS0} with the rescaled initial condition
$u_{0, \ld}(\bx) := \ld^{\frac{2}{p-1}} u_0 (\ld \bx)$.
We say that
the Sobolev index $s_c$ is critical
if
the homogeneous Sobolev norm  $\|\cdot\|_{\dot H^{s_c}(\R^d)}$
is invariant under this dilation symmetry.
In particular, the critical Sobolev index is given by $s_c = \frac{d}{2} - \frac{2}{p-1}$.
When $M\ne \R^d$,
we may not have this natural dilation symmetry.
Nonetheless,
the notion of the critical Sobolev index
provides us important heuristics.
In terms of the Cauchy problem \eqref{eq:nls1},
the
critical Sobolev index $s_c$
is given by
\begin{equation} \label{Introsc}
s_c = \frac d2 - \frac 1k.
\end{equation}

First, we state local well-posedness in
subcritical Sobolev spaces $H^s(\T^d)$ with $s > s_c$.

\begin{theorem}[Local well-posedness in subcritical spaces]\label{THM:3}
Let $d \geq 2$ and $k \in \mathbb{N}$.
Then, there exists $s_0 = s_0(k,d)$ such that the Cauchy problem \eqref{eq:nls1}
on a $d$-dimensional irrational torus $\T^d$ is locally well-posed in $H^{s}(\T^d)$
for $s > s_0$
in the following cases:


\smallskip
\noi
\textup{(i)}
$d=2$\textup{:}
\begin{itemize}
\item[(i.a)]  $k = 1$, $s_0= \frac{1}{3}$ \cite{CW},
\item[(i.b)] $k = 2, 3, 4, 5$, $s_0= \frac{7k - 3 }{7k+5}$,
\item[(i.c)] $k \geq 5$, $s_0= s_c = 1- \frac 1k$,
\end{itemize}

\noi
Note that the values of $s_0$ in \textup{(i.b)} and \textup{(i.c)} coincide when $k = 5$.

\smallskip
\noi
\textup{(ii)}
$d=3$\textup{:}
\begin{itemize}
\item[(ii.a)] $k = 1$, $s_0= \frac{2}{3}$ \cite{Bo4},
\item[(ii.b)] $k = 2$,  $s_0= \frac{53}{52}$,
\item[(ii.c)] $k \geq 3$, $s_0=s_c = \frac 32 - \frac 1k$,
\end{itemize}

\noi
\textup{(iii)}
$d\geq 4$\textup{:}
 $k\ge1$, $s_0=s_c = \frac d2 - \frac 1k$.
\end{theorem}

\noi
After Bourgain's seminal paper \cite{Bo2},
the Fourier restriction norm method,  involving the $X^{s, b}$-space,  has been applied
to study well-posedness of a wide class of equations.
In our proof, we also employ the $X^{s, b}$-spaces
and by the standard argument,
the proof is reduced to establishing certain multilinear Strichartz estimates.

\medskip

Furthermore, by applying the well-posedness theory involving
the $U^p$- and $V^p$-spaces developed by Tataru,
Koch,
and their collaborators \cite{KochT, HHK, HTT11, HTT2},
we prove some critical local well-posedness.\footnote{In a very recent paper, 
Strunk \cite{Strunk}
extended Theorem \ref{THM:4} to (i) $k \geq 3$ when $d = 2$ and (ii) $k = 2$ when $d = 3$.
The main idea in \cite{Strunk} is based on  considering Strichartz estimates in 
mixed Lebesgue spaces $L^q_tL^r_\bx$ to 
improve the {\it multilinear} Strichartz estimate (Proposition \ref{PROP:LWP2} below).
This clever argument avoids the need of 
 improving the scaling-invariant Strichartz estimate \eqref{P1}.}

\begin{theorem}[Local well-posedness in critical spaces]\label{THM:4}
Given $d \geq 2$ and $k \in \mathbb{N}$,
let
 $s_c$ be the critical Sobolev index given by  \eqref{Introsc}.
 Then,  the Cauchy problem \eqref{eq:nls1}
 on a $d$-dimensional irrational torus $\T^d$ is locally well-posed in the critical Sobolev space $H^{s_c}(\T^d)$ in the following cases:
\begin{itemize}
\item[(i)] $d=2$\textup{:}   $k\ge 6$,
\item[(ii)] $d=3$\textup{:}   $k\ge 3$,
\item[(iii)] $d\ge4$\textup{:} $k\ge 2$.
\end{itemize}
\end{theorem}

\noi
Once again, the proof is reduced to establishing
certain multilinear Strichartz estimates.
See Propositions \ref{PROP:XLWP} and \ref{PROP:LWP2}.

\medskip

Lastly, we briefly discuss the case of a {\it partially} irrational torus.
Namely, we consider Strichartz estimates
on an irrational torus $\T^d_{\pmb{\al}}$,  when
some of $\al_j$'s  in \eqref{Torus}
are rationally dependent.
In this case, we may obtain improvements
over  Theorem \ref{THM:2},
yielding better local well-posedness results than
those presented in Theorems \ref{THM:3} and \ref{THM:4}.
For simplicity of presentation,  we only
consider an example of the three-dimensional torus of the form
$\T^2 \times \T_{\al_3}$, where two periods are the same.
By a change of spatial variables as before,
we 
  consider the Cauchy problem \eqref{eq:nls1},
where
the multiplier $Q(\bn)$  in \eqref{eq:Q}
 is given by
\begin{align}\label{eq:Q1}
Q(\bn) =  n_1^2 + n_2^2 + \theta_3 n_3^2,
\quad \theta_3 >0,
\end{align}

\noi
i.e.~we set $\theta_1 = \theta_2 =1$.
Then, we have the following local well-posedness result
for the  energy-critical quintic NLS
on a three-dimensional partially irrational torus.

\begin{theorem}\label{THM:5}
Suppose that
%
$Q(\bn)$ is given by \eqref{eq:Q1}.
Then,
the energy-critical quintic NLS,  \eqref{eq:nls1} with $k = 2$,
on $\T^3$
is locally well-posed in the critical Sobolev space $H^1(\T^3)$.

\end{theorem}

\noi
Previously, Herr-Tataru-Tzvetkov \cite{HTT11}
proved local well-posedness in the energy space $H^1(\T^3)$ of the energy-critical
quintic NLS on the three-dimensional standard torus $\T^3$.
By combining the results in \cite{Bo5} and \cite{HTT2}, 
we also see that 
the energy-critical
cubic NLS on the four-dimensional standard torus $\T^4$
is local well-posedness in the energy space $H^1(\T^4)$.
See also the work by 
the third author  \cite{W}
for some other critical local well-posedness results.
The result in \cite{W}, however, does not cover an energy-critical setting.
As mentioned earlier, 
 Herr \cite{Herr} proved
 local well-posedness in the energy space
 of the energy-critical
quintic NLS on three-dimensional Zoll manifolds.
We point out that Theorem \ref{THM:5} seems to be the first local well-posedness result
of the energy-critical NLS in its energy space $H^1(\T^3)$, 
where there is no common minimal period for geodesics.

We present a sketch of the proof in Appendix \ref{SEC:B}.
More precisely, we revisit the argument in
Section \ref{SEC:level}
and prove the sharp Strichartz estimate \eqref{P1} on $\T^3$
for $p > \frac{14}{3}$
{\it under the assumption \eqref{eq:Q1}.}
The rest follows from a slight modification of the argument in
Section \ref{SEC:critical}.
Lastly, note that
Theorem \ref{THM:5} combined
with the conservation  of mass and Hamiltonian
yields
small data global well-posedness of
the quintic NLS,  \eqref{eq:nls1} with $k = 2$,  in $H^1(\T^3)$,
just as in \cite{HTT11, Herr}.
Recently, global well-posedness (for large data)
of the energy-critical quintic NLS on 
the three-dimensional standard torus $\T^3$
and on the three-dimensional  sphere $\mathbb{S}^3$ was obtained
by Ionescu-Pausader \cite{IP}
and Pausader-Tzvetkov-Wang \cite{PTW}, respectively.
It would be of interest to investigate if 
global well-posedness of  the energy-critical quintic NLS holds 
in the setting of 
Theorem \ref{THM:5}.\footnote{After Strunk's result \cite{Strunk},
it is now of interest to 
study global well-posedness of  the energy-critical quintic NLS 
on a {\it general} three-dimensional irrational torus
in its energy space $H^1(\T^3)$.}

\medskip

This paper is organized as follows.
In Section \ref{SEC:2}, we prove Theorem \ref{THM:2} when $ d\geq 3$
and partially when $ d = 2$,
via multilinear estimates and a duality argument.
In Section \ref{SEC:level}, we
establish certain level set estimates
and prove Theorem \ref{THM:2} when $ d = 2$.
In Section \ref{SEC:4}, we prove
local well-posedness results in subcritical Sobolev spaces
(Theorem \ref{THM:3}).
In Section \ref{SEC:critical}, we prove
local well-posedness results in critical Sobolev spaces
(Theorem \ref{THM:4}).
In Appendix \ref{SEC:A},
we present a proof of \eqref{A-1} below,
using the Hardy-Littlewood circle method.
In Appendix \ref{SEC:B}, we sketch a proof of Theorem \ref{THM:5}.

\medskip

\noi
{\bf Acknowledgments:} 
Z.~Guo is supported in part by NNSF  of China (No.11371037, No.11271023) and Beijing Higher Education Young Elite Teacher Project.
Y.~Wang is supported by NNSF of China (No.11126247, No.11201143) and AARMS Postdoctoral Fellowship.
The authors would like to thank the anonymous referee for thoughtful
comments that have significantly improved the introduction of this paper. 
T.~Oh would like to express his sincere gratitude to Professor~Harold N.~Shapiro
for his support and teaching in mathematics, including the Hardy-Littlewood circle
method used in this paper,
as well as in life.

\section{Strichartz estimates: Part 1}
\label{SEC:2}

In this section, we prove our main result (Theorem \ref{THM:2}) for $ d\geq3$
and present a partial proof for $ d = 2$.
In \cite{Bo4}, Bourgain treated the three-dimensional case.
His argument  is based on the following estimate:
\begin{equation}
\int_{\T} \Big| \sum_{0\le n\le N} e^{2\pi  i n^2 t}\Big|^{r}\,dt \sim N^{r-2},
\label{A-1}
\end{equation}

\noi
for $r > 4$.
We first apply this argument and generalize the result in \cite{Bo4}
 to a general dimension $d \geq 3$.
When $r = 4$, \eqref{A-1} holds with a logarithmic loss
(Hua's inequality).
This yields the endpoint case for $ d= 3, 4$.
When $d = 2$, this also proves Theorem \ref{THM:2} (ii)  but only for $p \geq 8$.
In Subsection \ref{SUBSEC:d2},
we present
a simple duality argument when $d = 2$.
This  proves
Theorem \ref{THM:2} (i)  for $p > 12$.
The full proof of Theorem \ref{THM:2} for $ d = 2$,
i.e.~\eqref{P1} for $ p > \frac{20}{3}$ and \eqref{P2} for $ p \geq \frac{20}{3}$,  is presented in
Section \ref{SEC:level}.

\subsection{Higher dimensional case: $d\ge 3$}
\label{SUBSEC:d3}

In this subsection, we prove Theorem \ref{THM:2} when $d \geq 3$.
First, we prove the following lemma,
which can be viewed as a version of Hausdorff-Young's inequality.
\begin{lemma}[Hausdorff-Young's inequality]
\label{LEM:HY}
Let $d \geq 2$ and $\ba \in \Z^d$.
Given a sequence $\{c_\bn\}_{\bn\in \Z^d}$,
define $F_\ba (t)$ by
\begin{equation}
F_\ba(t) = \sum_{\bn\in \Z^d} c_\bn c_{\ba-\bn} e^{2 \pi i[Q(\bn) + Q(\ba-\bn)]t},
\label{A0}
\end{equation}

\noi
where $Q(\bn)$ is as in \eqref{eq:Q}.
Then, for $p \geq 2$, we have 
\begin{equation}
\|F_\ba(t)\|_{L^p_t([-1,1])} \les
\Bigg[\sum_{k\in \Z} \bigg(\sum_{|Q(\bn) + Q(\ba-\bn)-k|\le \frac12} |c_\bn c_{\ba-\bn}| \bigg)^{p'}\Bigg]^{\frac1{p'}},
\label{A1}
\end{equation}
where $\frac 1p+\frac{1}{p'} =1$.
\end{lemma}

\noi
Lemma  \ref{LEM:HY} was used in \cite{Bo4} for the three-dimensional case.
See also Lemma 2 in \cite{CW}.
A proof for  general dimensions
in  \cite{CW} relies on Schur's lemma.
In the following,  we give a direct proof  for reader's convenience.

\begin{proof}
When $p = \infty$,
\eqref{A1} follows immediately.
Hence, by interpolation,
 it suffices to prove \eqref{A1} for  $p=2$.
 Let $\eta(t)$ be a  cutoff function supported on $[-2,2]$
 such that $\eta \equiv 1$ on $[-1,1]$.
 By Plancherel identity,  we have
\begin{align*}
 \|F_\ba(t)\|_{L^2_t([-1,1])}
 & \le\|F_\ba(t) \eta(t)\|_{L^2_t(\R)} \\
& =  \bigg\|\sum_\bn c_\bn c_{\ba-\bn} \ft\eta\big(\tau-[Q(\bn) + Q(\ba-\bn)]\big)\bigg\|_{L^2_\tau}\\
& =  \bigg\|\sum_{k\in \Z} \sum_{n\in I_{k,\ba}} c_\bn c_{\ba-\bn}
\ft\eta\big(\tau-[Q(\bn) + Q(\ba-\bn)]\big)\bigg\|_{L^2_\tau}\\
& =  \bigg\|\sum_{k\in \Z} B_k(\tau)\bigg\|_{L^2_\tau},
\end{align*}

\noi
where $I_{k,\ba} = \big\{\bn\in \Z^d: \,  Q(\bn) + Q(\ba-\bn)-k  \in(-\frac 12,\frac 12]\big\}$
and
 \[  B_k(\tau) = \sum_{\bn\in I_{k,\ba}} c_\bn c_{\ba-\bn} \widehat\eta\big(\tau-[Q(\bn) + Q(\ba-\bn)]\big).\]

\noi
Noting that
$|B_k(\tau)| \les \sum_{n\in I_{k,\ba}} |c_\bn c_{\ba-\bn}| \jb{\tau-k}^{-2}$,
we have
\begin{align*}
  \Big\|\sum_{k\in \Z} & B_k(\tau) \Big\|^2_{L^2_\tau}
 =  \Big\|\Big(\sum_{k\in \Z} B_k(\tau)\Big)^2 \Big\|_{L^1_\tau}
 \les  \sum_{k, k'} \|B_{k}(\tau)B_k'(\tau)\|_{L^1_\tau}\\
& \les
\sum_{k, k'}
\sum_{\bn\in I_{k,\ba}}  \sum_{\bn'\in I_{k',\ba}} |c_\bn c_{\ba-\bn}|  |c_{\bn'} c_{\ba-\bn'}|
\int_{\R}\jb{\tau-k}^{-2} \jb{\tau-k'}^{-2} d\tau
\\
& \les
\sum_{k, k'}
\frac 1{\jb{k-k'}^2}
\sum_{\bn\in I_{k,\ba}}
 \sum_{\bn'\in I_{k',\ba}} |c_\bn c_{\ba-\bn}|  |c_{\bn'} c_{\ba-\bn'}|
\intertext{By Cauchy-Schwarz inequality (in $k$) followed by Young inequality, we have}
&  \leq   \bigg[\sum_k \Big(\sum_{\bn\in I_{k,\ba}} |c_\bn c_{\ba-\bn}|\Big)^2\bigg]^{\frac12}
\bigg[\sum_k \Big( \sum_{k'}\sum_{\bn'\in I_{k',\ba}}
\frac{|c_{\bn'} c_{\ba-\bn'}|}{\jb{k-k'}^2}\Big)^2\bigg]^{\frac12}\\
&  \les  \sum_k \bigg(\sum_{\bn\in I_{k,\ba}} |c_\bn c_{\ba-\bn}|\bigg)^2.
\end{align*}

\noi
This completes the proof of Lemma \ref{LEM:HY}.
\end{proof}

Next,  we
state the main proposition.
 Theorem \ref{THM:2} then follows this proposition and Bernstein's inequality
 when $d \geq 3$.

%

\begin{proposition}\label{PROP:Str d3}
Let $f$ be a function on $\T^d$ with
 $\supp \ft f \subset [-N,N]^d$.

\noi
\textup{(i)}
Let  $d\ge 3$. Then, for $p\ge \max\big(\frac {16}d+, 4\big)$, we have
\begin{align}\label{Str d3}
\|e^{-it\Delta}f\|_{L^p_{t, \textup{loc}}L^4_{\bx}}
\les N^{\frac d4 -\frac 2p} \|f\|_{L^2},
\end{align}

\smallskip
\noi
\textup{(ii)}
Suppose that $d$ and $p$ satisfy
\textup{(ii.a)} $d = 2$, $p \geq 8$,
\textup{(ii.b)}
$(d, p) = (3, \frac{16}{3})$,
or
\textup{(ii.c)}
$(d, p) = (4, 4)$.
Then, we have
\begin{align}\label{Str d3a}
\|e^{-it\Delta}f\|_{L^p_{t, \textup{loc}}L^4_{\bx}}
\les N^{\frac d4 -\frac 2p}
(\log N)^\frac{2}{q}
\|f\|_{L^2},
\end{align}

\noi
where $q = p$ when $d = 3, 4$ and $q = 8$ when $d = 2$.

\end{proposition}

\noi

\noi
 Bourgain  proved \eqref{Str d3} for $ d= 3$.
 See Proposition 1.1 in \cite{Bo4}.
Our proof follows  the ideas developed in \cite{Bo4}.
 By setting $p = 4$ when $ d \geq 5$
 and $p = \frac{16}{d}+$ when $ d= 3, 4$,
 Proposition \ref{PROP:Str d3}
yields the $L^4$-Strichartz estimate,
which
 improves  the result in \cite{CW} for $ d\geq 3$.
Note that
 the Strichartz estimate \eqref{Str d3} is essentially sharp in higher dimensions ($d \geq 4$).
Indeed, on $\R^d$ with $ d \geq 2$,
by Sobolev inequality and interpolation of \eqref{P3} and \eqref{P4}
\begin{align*}
\|e^{-it\Delta}f\|_{L^p_{t}L^4_{x} (\R\times \R^d)}
\les
N^{2(\frac 14 -\frac 1p)}\|e^{-it\Delta}f\|_{L^4_{t, x} (\R\times \R^d)}
\les  N^{\frac d4 -\frac 2p} \|f\|_{L^2(\R^d)},
\end{align*}

\noi
for $p \geq 4$.

We first use Proposition \ref{PROP:Str d3}
to prove Theorem \ref{THM:2} when $d \geq 3$.

\begin{proof}[Proof of Theorem \ref{THM:2} for $d \geq 3$]
Suppose that $p \geq \max(\frac{16}{d}+, 4)$,  satisfying the hypothesis of Theorem \ref{THM:2} (i).
By Bernstein's inequality and Proposition \ref{PROP:Str d3} (i), we have
\[
\|e^{-it\Delta} f\|_{L^p_{t,\bx}} \les N^{\frac d4- \frac dp} \|e^{-it\Delta} f\|_{L^p_tL^4_{\bx}}
\les N^{\frac d2-\frac {d+2}p} \|f\|_{L^2}.
\]

\noi
By repeating the same argument with Proposition \ref{PROP:Str d3} (ii),
we obtain
Theorem \ref{THM:2} (ii)
when $d = 3, 4$.
When $d = 2$, this yields Theorem \ref{THM:2} (ii) only for $p \geq 8$.
\end{proof}

We now present the proof of Proposition \ref{PROP:Str d3}.

\begin{proof}[Proof of Proposition \ref{PROP:Str d3}]
(i)
Let  $Q(\bn)$ be as in \eqref{eq:Q}.
Then,  we have
\begin{equation}
\label{A1a}
(e^{-it\Delta} f )(\bx) = \sum_{\bn\in \Z^d} \ft f (\bn) e^{2\pi  i (\bn\cdot \bx + Q(\bn) t)}.
\end{equation}

\noi
With $c_\bn = \ft f (\bn)$,
let $F_\ba(t)$ be as in \eqref{A0}.
Then,
by Minkowski's integral inequality with $p \geq 4$, we have
\begin{align}
\label{A2}
 \|e^{-it\Delta}f\|^2_{L^p_t L^4_{\bx}}
&  = \|(e^{-it\Delta} f)^2\|_{L^{\frac p2 }_t L^2_{\bx}}
=  \bigg\|\Big(\sum_{\ba\in \Z^d}
|F_\ba(t)|^2\Big)^\frac{1}{2}
\bigg\|_{L^{\frac p 2}_t} \nonumber \\
&  \le
 \bigg(\sum_{\ba\in \Z^d}
\|F_\ba(t)\|_{L^\frac{p}{2}_t}^2
\bigg)^{1/2}.
\end{align}

%


For $\l \in \Z$,
let
 $A_\l = \{\bn\in \Z^d :\,
 |Q(\bn)-\l| \le 1\} \cap [-N, N]^d$. 
Noting that $Q(\bn) + Q(\ba- \bn) = \frac 12 \big(Q(2\bn-\ba) + Q(\ba)\big)$,
the condition $|Q(\bn) + Q(\ba-\bn) - k|\le \frac 12$ is equivalent to $2\bn\in \ba + A_\l$
with $\l = 2k- Q(\ba)$.
Note that, $\big|\{ \l \in \Z:\, 2\bn \in \ba +A_\l\}\big| \les 1$
for all $\bn \in \Z^d$.
Then, by
Lemma \ref{LEM:HY} and Cauchy-Schwarz and H\"older's  inequalities, we have
\begin{align}\label{A3}
\|F_\ba\|_{L^{\frac p2}_t}
& \le
\Bigg[\sum_{\l \in \Z} \bigg(\sum_{2\bn\in \ba + A_\l} |c_\bn c_{\ba- \bn}| \bigg)^{\frac p{p-2}}
\Bigg]^{\frac{p-2}p} \notag \\
 &\les \Bigg[\sum_\l |A_\l|^{\frac p{2(p-2)}}
 \bigg(\sum_{2\bn\in \ba + A_\l} |c_\bn c_{\ba-\bn}|^2 \bigg)^{\frac p{2(p-2)}} \Bigg]^{{\frac{p-2}p}} \nonumber\\
& \leq   \bigg(\sum_\l |A_\l|^{\frac p{p-4}}\bigg)^{\frac{p-4}{2p}}
\bigg( \sum_{\l } \sum_{2\bn\in \ba + A_\l}
|c_\bn c_{\ba-\bn}|^2 \bigg)^{\frac 12} \notag \\
& \sim   \bigg(\sum_\l |A_\l|^{\frac p{p-4}}\bigg)^{\frac{p-4}{2p}}
\bigg(  \sum_{\bn \in \Z^d }
|c_\bn c_{\ba-\bn}|^2 \bigg)^{\frac 12}.
\end{align}

\noi
From \eqref{A2} and \eqref{A3}, we have
\begin{equation}
\label{A4}
\|e^{-it\Delta} f\|_{L^p_tL^4_{\bx}} \le C
\bigg(\sum_\l |A_\l|^{\frac p{p-4}}\bigg)^{\frac{p-4}{4p}} \|f\|_{L^2}.
\end{equation}

Now, let
 $\eta(t)$ be a smooth function with a compact support $I \subset \R$
 such that $\ft \eta \geq 0$ and $\ft \eta \geq 1$ on $[-1, 1]$.
Now we estimate $\Big(\sum_\l |A_\l|^{\frac p{p-4}}\Big)^{\frac{p-4}{4p}}$, using
\[
|A_\l|\le \int \Big[\sum_{\substack{\bn\in\Z^d\\|n_j|\le N}} e^{2\pi iQ(\bn)t}\Big]\eta(t)e^{-2
\pi i\l t} dt.
\]

If $p\le 8$, then we have $\frac p{p-4}\ge 2$.
Then,  by Hausdorff-Young's inequality, we have
\begin{align}\label{A5}
\Big(\sum_\l |A_\l & |^{\frac p{p-4}}\Big)^{\frac{p-4}{4p}}
 \les
\bigg[\int_{I}\prod_{j=1}^d \Big| \sum_{ |n_j|\le N}
e^{ 2\pi i\theta_j n_j^2 t}\Big|^{\frac p4}\,dt\bigg]^{\frac1p} \nonumber\\
&  \les
\prod_{j=1}^d
\bigg[\int_{I} \Big| \sum_{ |n_j|\le N} e^{2\pi  i\theta_j n_j^2 t}\Big|^{\frac {dp}4}\,dt\bigg]^{\frac1{dp}}
 \les
\bigg[\int_{I} \Big| \sum_{ 0 \leq n\le N} e^{2\pi  i n^2 t}\Big|^{\frac {dp}4}\,dt\bigg]^{\frac1p}.
\end{align}

\noi
Note that $r= \frac {dp}4 >4$, since $p> \frac{16}d$.
Then, by an  application of the Hardy-Littlewood circle method
(see Appendix \ref{SEC:A}),
we have
\begin{equation}
\int_{I} \Big| \sum_{0\le n\le N} e^{2\pi  i n^2 t}\Big|^{r}\,dt \sim N^{r-2},
\label{A6}
\end{equation}

\noi
yielding
$\eqref{A5}\les N^{\frac d4-\frac 2p}$.
Hence, \eqref{Str d3} follows from \eqref{A4} in this case.

If  $p>8$, then
by  Bernstein's inequality (in $t$), 
we have
\begin{equation}\label{A7}
\|e^{-it\Delta} f\|_{L^p_tL^4_{\bx}} \le C N^{\frac 14-\frac 2p} \|e^{-it\Delta} f\|_{L^8_tL^4_{\bx}}.
\end{equation}

\noi
Then, \eqref{Str d3} follows from \eqref{A7}
and \eqref{Str d3} for $p = 8$.

\smallskip

\noi
(ii)
When $(d, p) = (2, 8),$ $(3, \frac{16}{3})$, or $(4, 4)$, we
have $r = \frac{dp}{4} = 4$.
In this case,  \eqref{A6} does not hold.
Nonetheless, by Hua's inequality \cite{V}, we have
\begin{equation}
\int_{I} \Big| \sum_{0\le n\le N} e^{2\pi  i n^2 t}\Big|^{4}\,dt \les  N^{2} (\log N)^2.
\label{A8}
\end{equation}

\noi
See also \cite[(8.13)]{IK}.
Then, \eqref{Str d3a} follows from \eqref{A8} and  repeating the computation in (i).
This completes the proof of Proposition \ref{PROP:Str d3}.
\end{proof}


\subsection{Two dimensional case}\label{SUBSEC:d2}
For $d=2$, the sharp estimate \eqref{P1} is not covered by
Proposition \ref{PROP:Str d3}.
In the following, we use a simple duality argument
and  prove the sharp Strichartz estimate \eqref{P1} for $p>12$.
Without loss of generality, we  assume that
\begin{align}\label{B1}
 Q(\mathbf{n})=  n_1^2 +\theta  n_2^2, \quad \tfrac{1}{C} \leq \theta \leq C.
\end{align}

\noi
Then,
the local-in-time Strichartz estimate
can be expressed as 
\begin{equation}\label{B2}
\bigg\|\sum_{\mathbf{n}\in S_{N}}a_\mathbf{n}
e^{2\pi i(\mathbf{n}\cdot\mathbf{x}+Q(\mathbf{n})t)}\bigg\|_{L^p_{t, \bx}(I\times \T^2)}\leq K_{p,N}\bigg(\sum_{\bn\in S_N}|a_\mathbf{n}|^2\bigg)^{1/2},
\end{equation}

\noi
where $S_{N}$ denotes the following set:
\begin{equation}\label{B2a}
S_N: =  \big\{(n_1,n_2)\in \mathbb{Z}^2: |n_j|\leq N,\, j=1,2 \big\}.
\end{equation}

\noi
Our task is to seek for an optimal constant $K_{p, N}$.
By duality, \eqref{B2} is equivalent to
\begin{equation}\label{B3}
\bigg(\sum_{\bn\in S_N}\big|\widehat{f}(\mathbf{n}, Q(\mathbf{n}))\big|^2\bigg)^\frac{1}{2}
\leq K_{p, N}
\|f\|_{L^{p'}_{t, \bx}(I\times \T^2)},
\end{equation}

\noi
for any $f \in L^{p'}(I\times \T^2)$,
where $\frac 1 p+  \frac 1{p'} = 1$.
Here,
the Fourier transform $\ft f$ is defined by
\[
\widehat f(\bn,\tau) = \int_\R \int_{\T^2} e^{-2\pi i \bn\cdot\bx} e^{-2\pi i  \tau t } \ind_{I}(t) f(\bx,t)\,d\bx\,dt.
\]

%

\noi
Then,  \eqref{P1} for $p > 12$  follows once we prove the next proposition.

\begin{proposition}\label{PROP:Str d2} For $p> 12$,
we have $K_{p, N} \les N^{1-\frac 4p}$.
Namely,  we have
\begin{equation}
\bigg(\sum_{\bn \in S_N }
\big|\widehat{f}(\bn, Q(\bn))\big|^2 \bigg)^\frac{1}{2}
\les N^{1-\frac{4}{p}} \|f\|_{L^{p'}_{t, \bx}(I\times \T^2)}^2.
\label{B4}
\end{equation}
\end{proposition}

\begin{remark}\rm
Recall that, on the standard torus $\T^2$,
i.e.~with $Q(\bn) = n_1^2 + n_2^2$,
Bourgain \cite{Bo2} proved $K_{p, N} \les N^{1-\frac 4p}$ for $ p > 4$.
Hence, Proposition \ref{PROP:Str d2} states that, on an irrational torus,
the same estimate for $K_{p, N}$ holds, but only for $p > 12$.

\end{remark}


\begin{proof}
Without loss of generality, assume that $I$ is centered at 0.
Let $\RR$ be a kernel defined by
\begin{equation}\label{B4a}
\mathbf{R}(\bx, t) =
\sum_{\bn \in S_N}  e^{2\pi i (\bn \cdot \bx +  Q(\mathbf{n}) t)}.
\end{equation}

\noi
Then, defining $R_\theta$ by
\begin{equation}
R_\theta(x, t) =
\sum_{|n| \leq N}  e^{2\pi i (nx  +  \theta n^2  t)},
\label{B4b}
\end{equation}

\noi
we have
$\mathbf{R}(\bx, t) = R_1(x_1, t)R_\theta(x_2, t).$
%
%
From Proposition 3.114 in \cite{Bo2}, we have
\begin{align}
\|R_1(x, t)\|_{L^p_{t, x}(I\times \T)}\leq C_{p, I}N^{1- \frac 3p},
\label{B5}
\end{align}

\noi
for $p > 6$.
Bourgain's argument is based on an application of the Hardy-Littlewood circle method.
See also Lemma 2.4 in \cite{Hu-Li}
for a simpler proof based on the Poisson summation formula.
Note that \eqref{B5} does not hold for $p = 6$.
See Rogovskaya \cite{R} and \cite{Bo2}.

By H\"older's inequality, \eqref{B5},
and Sobolev inequality, we have
\begin{align}
\|\mathbf{R}\|_{L^p_{t, \bx}(I\times \T^2)}
& =  \bigg(\int_I \|R_1(x_1, t)\|^p_{L^p_{x_1}}
\| R_\theta (x_2,t)\|^p_{L^p_{x_2}}\,dt\bigg)^\frac{1}{p} \notag \\
& \le \|R_1(x_1,t)\|_{L^p_{t,x_1}}
  \|R_\theta(x_2,t)\|_{L^\infty_t(I;  L^p_{x_2})}  \notag \\
& \les N^{1 - \frac 3p} \|R_\theta(x_2,t)\|_{L^\infty_t(I;  H^{\frac{1}{2}-\frac 1p}_{x_2})}
\les N^{2 - \frac 4 p}.
\label{B6}
\end{align}

\noi
By \eqref{B4a}, Young's inequality, and \eqref{B6}, we have
\begin{align*}
 \sum_{\bn\in S_N }
 \big| \widehat {f}(\bn\,,  Q(\bn))\big|^2
& =
\jb{ \mathbf{R}  * \ind_I f, \ind_I f }
\leq \|\mathbf{R}\|_{L^\frac{p}{2}_{t, \bx}(2I \times \T^2)}\|f\|_{L^{p'}_{t, \bx}(I\times \T^2)}^2\\
& \les N^{2-\frac{8}{p}}\|f\|_{L^{p'}_{t, \bx}(I\times \T^2)}^2
\end{align*}

\noi
as long as  $p > 12$.
\end{proof}

\section{Strichartz estimates: Part 2}
\label{SEC:level}

\subsection{Level set estimates} 
In this section, we prove
Theorem \ref{THM:2} when $d = 2$.
%
The main ingredient is the level set estimates
on irrational tori
in Proposition \ref{PROP:level} below.
For level sets estimates on the usual torus $\T^d$, see \cite{Bo2, Hu-Li}.
It turns out that
these level set estimates are useful only when $d = 2, 3$
(see Remark \ref{REM:level}),
but we state and prove the results for a general dimension.
In the following, we assume that  $\theta_1 = 1$ in \eqref{eq:Q}
for simplicity.
Namely, we consider
\begin{equation}
 Q(\bn) = n_1^2 + \theta_2 n_2^2 + \cdots + \theta_d n_d^2.
\label{CQ}
 \end{equation}

\noi
Also, let $S_N = \{ \bn \in \Z^d: \, |n_j| \leq N, \, j = 1, \dots, d \}$.

\begin{proposition}\label{PROP:level}

Let $I$ be a compact interval in $\R$.
Given
\begin{equation}
f (\bx) = \sum_{\bn \in S_N} c_\bn e^{2\pi i \bn \cdot \bx}
\label{C0}
\end{equation}

\noi
such that $\|c_\bn\|_{\l^2_\bn} = 1$,
define  the distribution function $A_\ld$
by
\begin{equation}
 A_\ld = \big\{ (\bx, t) \in \T^2 \times I: \, \big|\big(e^{-it \Dl} f\big)(\bx, t)\big| > \ld\big\}.
\label{C00a}
\end{equation}

\noi
\textup{(i)}
For any $\eps > 0$, we have
\begin{equation}
|A_\ld |
\les N^{2(d-1)\frac{1}{1+6\eps} } \ld^{-6 + \frac{24}{1+6\eps}\eps}
\label{C1}
\end{equation}

\noi
for $\ld \ges N^{\frac{d}{2} - \frac{1}{4}+\eps}$.

\medskip
\noi
\textup{(ii)} 
Let  $ q > 6$.
Then, there exists small $\eps >0$ such that
\begin{equation}
|A_\ld |\les N^{\frac{d}{2} q - (d+2)} \ld^{-q}
\label{E0}
\end{equation}

\noi
for $\ld \ges N^{\frac{d}{2}-\eps}$.


In \eqref{C1} and \eqref{E0},
 the implicit constants depend on $\eps > 0$, $q > 6$, and $|I|$,
but are independent of $f$.

\end{proposition}


We present the proof of Proposition \ref{PROP:level} in
Subsections \ref{SUBSEC:PROP1} and \ref{SUBSEC:PROP2}.
In the following, we
use
Proposition \ref{PROP:level}
to prove
Theorem \ref{THM:2} when $ d = 2$.
First, we present the proof of
Theorem \ref{THM:2}  (ii.a),
i.e.~we  prove \eqref{P2} for $p \geq   \frac{20}{3}$ when $ d= 2$.
Recall that  Catoire-Wang \cite{CW} proved
\begin{equation}
\|e^{-it \Dl} f\|_{L^4_{t, \bx}(I \times \T^d)}
\les N^\frac{1}{6} \|f\|_{L^2(\T^d)}
\label{C0a}
\end{equation}

\noi
for $f \in \T^2$ with $\supp \ft{f} \subset [-N, N]^2$.
Given $f$ as in \eqref{C0},
let $F(\bx, t) = e^{-it \Dl} f(\bx, t)$.
By Cauchy-Schwarz inequality, we have
$\|F\|_{L^\infty_{t, \bx}} \les N. $
Then, with Proposition \ref{PROP:level} (i) and \eqref{C0a},
we have
\begin{align*}
\int_{I \times \T^2 } |F(\bx, t)|^p d\bx dt
& \leq \int_{ N^{\frac34+\eps}\les |F| \les N}|F(\bx, t)|^p d\bx dt
+ N^{(\frac{3}{4}+\eps)(p - 4)}\int |F(\bx, t)|^4 d\bx dt\\
& \les N^{2-\frac{12}{1+6\eps}\eps} \int_{N^{\frac34+}}^N\ld^{p - 7+ \frac{24}{1+6\eps}\eps} d\ld
+ N^{(\frac{3}{4}+\eps)(p - 4)+ \frac 23}\\
& 
 \les N^{p - 4+},
\end{align*}

\noi
where the last inequality holds as long as $p \geq \frac{20}{3}$.
This proves Theorem \ref{THM:2} (ii.a).
By Proposition \ref{PROP:level} (i) and (ii),
 Theorem \ref{THM:2} (i.a)
follows in a similar manner.
We omit details.

\begin{remark} \rm
When $d = 2$,
Proposition \ref{PROP:level} (i) and (ii)
basically says that the level set estimates \eqref{C1} and \eqref{E0}
are sufficient in proving the Strichartz estimates \eqref{P1} and \eqref{P2}
for $p > 6$
as long as
$\ld$ is  {\it large}: $\ld \geq N^{\frac{1}{4}+}$.
Hence, an improvement on Theorem \ref{THM:2} when $ d= 2$ may be obtained
if we can improve the lower bound on $\ld$
in Proposition \ref{PROP:level} (i)
or the $L^4$-Strichartz estimate \eqref{C0a}.
\end{remark}

\begin{remark}\label{REM:level} \rm
%
In  \cite{Bo5}, Bourgain proved
\begin{equation}
\|e^{-it \Dl} f\|_{L^p_{t, \bx}(I \times \T^d)}
\les N^\eps \|f\|_{L^2(\T^d)}
\label{C0b}
\end{equation}

\noi
for $ p = \frac{2(d+1)}{d}$.
See Proposition 8 and the comment afterward in \cite{Bo5}.
%
%
Combining
Proposition \ref{PROP:level} (i)
and \eqref{C0b}, we obtain \eqref{P2}
only for $ p \geq \frac{2(3d+1)}{d}$.
When $d = 2$, the combination of
Proposition \ref{PROP:level} (i)
and \eqref{C0a} yields a better result.
When $d \geq 3$, 
Proposition \ref{PROP:Str d3}
yields better results.
We point out that,
when $d = 3$, combining
Proposition \ref{PROP:level} (i) with Theorem \ref{THM:1} (ii)
yields another proof of Theorem \ref{THM:2} (ii.b).
\end{remark}

%
%
%
%

%
%

\subsection{Proof of Proposition \ref{PROP:level} (i)} \label{SUBSEC:PROP1}

Let $\eta$ be a smooth cutoff function supported on $[\frac{1}{200}, \frac{1}{100}]$.
Given $q \in \mathbb{N}$,
define  $J_q$ by
\begin{equation}
 J_q = \{ a \in \mathbb{N}:\, 1\leq a \leq q, \, (a, q) = 1\}.
 \label{C1a}
\end{equation}

\noi
Then, for given $M \in \mathbb{N}$ with $M \geq N$, we define
\[ \Phi(t) = \sum_{M\leq q < 2M}\sum_{a \in J_q}
\eta\Big( q^2 \big\| t - \tfrac{a}{q}\big\|\Big),\]

\noi
where $\|x\| = \min_{n\in \Z}|x - n|$ denotes the distance
of $x$ to the closest integer.
Note that $\Phi$ is periodic with period 1.
By taking a Fourier transform, we have
\begin{equation}
 \ft \Phi(k) =
\sum_{M\leq q < 2M}\frac{1}{q^2}
c_q(k) \, \ft \eta(q^{-2} k),
\label{C2}
\end{equation}

\noi
where $c_q(k)$ denotes
Ramanujan's sum: $c_q(k) : =\sum_{a \in J_q} e^{-2\pi i \frac a qk}$.
Let $\phi(q)$ be  the Euler's totient  function defined by $\phi(q) = \sum_{a \in J_q} 1$.
Then, by Theorem 330 in \cite{HW}, we have
\begin{equation*}
 \ft \Phi(0) \sim
\frac{1}{M^2} \sum_{M\leq q < 2M} \phi(q) \sim 1.
\end{equation*}

\noi
Namely, $\ft\Phi(0)$ is independent of $M$.

Without loss of generality, assume that $I $ is centered at $0$.
With $Q(\bn)$ in \eqref{CQ},
define $\mathbf{R}$ as in \eqref{B4a},
where $S_N = \{ \bn \in \Z^d: \, |n_j| \leq N, \, j = 1, \dots, d \}$.
Then, we have
$\mathbf{R}(\bx, t) = R_1(x_1, t)\prod_{j = 2}^d R_{\theta_j}(x_j, t),$
where $R_\theta$ is defined in \eqref{B4b}.
Now,
letting $\chi$ be a smooth cutoff function support on $3I$
such that $\chi(t) \equiv 1$ on $2I$,
define $\mathbf{R}_1$ and $\mathbf{R}_2$
by
\begin{equation}
\mathbf{R}_1(\bx, t) =   \frac{\Phi(t)}{\ft \Phi(0)} \RR(\bx, t) \chi(t) \quad
\text{and} \quad
\RR_2 (\bx, t) = \RR (\bx, t)\chi(t) - \RR_1 (\bx, t) .
\label{C2a}
\end{equation}

\noi
Noting that the intervals
$ I_{\l, q, a}: = \big[ \l + \frac aq + \frac{1}{200 q^2}, \l+ \frac aq + \frac{1}{100 q^2}\big] $
are disjoint for distinct values of $\l, a$, and $q \sim M \gg 1$,
it follows from
Weyl's inequality \cite[Theorem 1 on p.~41]{M}
that
\begin{equation}
 |R_1(x_1, t) | \les
\frac{N}{q^\frac{1}{2}} + N^\frac{1}{2}(\log q)^\frac{1}{2} + q^\frac{1}{2}(\log q)^\frac{1}{2}
\les M^\frac{1}{2} (\log M)^\frac{1}{2}
\label{C2b}
\end{equation}

\noi
for $t \in I_{\l, q, a}$ since $q\sim M \geq N$.
Then, along with a trivial bound
$ |R_{\theta_j}(x_j, t) | \les N$,
we obtain
\begin{equation}
\|\RR_1\|_{L^\infty_{t, \bx}}
\les \min\big( N^{d-1} M^\frac{1}{2} (\log M)^\frac{1}{2}, N^d\big).
\label{C3}
\end{equation}

Next, we consider $\RR_2$.
By expanding $\Phi(t)$ in the Fourier series, we have
\begin{align}
\RR_2(\bx, t) = - \frac{1}{\ft \Phi(0)} \RR(\bx, t) \chi(t) \sum_{k \ne 0} \ft\Phi(k) e^{2\pi i k t}.
\label{C3a}
\end{align}

\noi
First, recall the following lemma
(Lemma 3.33 in \cite{Bo2}).
Given $M \in \mathbb{N}$
and $k \in \Z$, we have
\begin{equation}
\sum_{M\leq q < 2M} |c_q(k)| 
\les d(k, M)M^{1+},
\label{C4}
\end{equation}

\noi
where $d(k, M)$ denotes the number of divisors of $k$ less than $M$.
Then,
by taking a Fourier transform of \eqref{C3a} with \eqref{C2}, \eqref{C4},
and $d(k, M) \les k^{0+}$, we have
\begin{align}
|\ft \RR_2(\bn, \tau) |
& =  \bigg| \ind_{S_N}(\bn)\sum_{k\ne0} \frac{\ft \Phi(k)}{\ft \Phi(0)}
\, \ft \chi (\tau - Q(\bn) - k)\bigg|\notag \\
& \les
\frac{1}{M^2} \sum_{k\ne0}
\sum_{M\leq q < 2M}|c_q(k)|
\big|\ft \eta(q^{-2} k)
 \ft \chi (\tau - Q(\bn) - k)\big|\notag \\
 & \les
\frac{1}{M^2} \sum_{k\ne0}
k^{0+} M^{1+} \Big(\frac{M^2}{k}\Big)^{0+}
\frac{1}{\jb{\tau - Q(\bn) - k}^{10}}\notag \\
 & \les M^{-1 + }.
 \label{C5}
\end{align}

Define $\Theta_\ld(\bx, t)$ by
\begin{equation}
\Theta_\ld(\bx, t) = \exp\big(i \arg(e^{-it \Dl}f(\bx))\big) \cdot \ind_{A_\ld}(\bx, t).
\label{C5a}
\end{equation}

\noi
Note that
 $\supp \Theta_\ld(\bx, \cdot) \subset I$
 for each $\bx \in \T^d$.
Then, by Cauchy-Schwarz inequality with \eqref{C0}, we have
\begin{align}
\ld^2 |A_\ld|^2
& \leq \bigg(\int_{I\times \T^2} \big(e^{-it \Dl} f\big)(\bx, t) \cj{\Theta_\ld(\bx, t)} d\bx dt\bigg)^2
= \bigg(\sum_{\bn \in S_N} c_\bn \cj{\ft \Theta_\ld (\bn, Q(\bn))}\bigg)^2 \notag \\
& \leq \sum_{\bn \in S_N} \big| \ft \Theta_\ld (\bn, Q(\bn))\big|^2
= \big\langle \RR* \Theta_\ld, \Theta_\ld\big\rangle
= \big\langle (\RR\chi)* \Theta_\ld, \Theta_\ld\big\rangle \label{C5b} 
\intertext{By \eqref{C2a}, \eqref{C3}, and \eqref{C5}, we have}
& = \big\langle \RR_1* \Theta_\ld, \Theta_\ld\big\rangle
+ \big\langle \RR_2* \Theta_\ld, \Theta_\ld\big\rangle  \notag \\
&  \leq \|\RR_1\|_{L^\infty_{t, \bx}} \|\Theta_\ld\|_{L^1_{t, \bx}}^2
+ \|\ft \RR_2\|_{L^\infty_\tau \l^\infty_\bn } \|\Theta_\ld\|_{L^2_{t, \bx}}^2  \notag \\
& \leq C_1 N^{d-1}  M^{\frac{1}{2}+\eps_1} |A_\ld|^2
+ M^{-1+\eps_2} |A_\ld|
\label{C6}
\end{align}

\noi
for small $\eps_1, \eps_2 > 0$.

Now,  choose $M \geq N$ such that
\begin{equation}
N^{d-1} M^{\frac{1}{2}+\eps_1} \sim \ld^2.
\label{C7}
\end{equation}

\noi
The condition \eqref{C7} with $M \geq N$
implies that $\ld \ges N^{\frac{d}{2} - \frac{1}{4} + \frac{\eps_1}{2}}$.
Then,
\eqref{C6} yields
\begin{align*}
|A_\ld| \les \Big(\frac{N^{2(d-1)}}{\ld^4}\Big)^\frac{1 - \eps_2}{1+2\eps_1} \ld^{-2}
\les N^{2(d-1)\frac{1}{1+3\eps_1} } \ld^{-6+\frac{12}{1+3\eps_1}\eps_1} 
\end{align*}

\noi
by setting $\eps_2 = \eps_2(\eps_1)$ such that
$\frac{1 - \eps_2}{1+2\eps_1} = \frac{1}{1+3\eps_1}$.
This proves \eqref{C1} with
$\eps =  \frac{\eps_1}{2}$.

\subsection{Proof of Proposition \ref{PROP:level} (ii)}
\label{SUBSEC:PROP2}

In this subsection, we prove the level set estimate \eqref{E0},
which is sharp
for $ q > 6.$
The following argument is inspired by Bourgain's paper \cite{Bo2}.
We first  go over some basic setups,
restricting our attention to $t \in \T$.

Let $\{\s_n\}_{n\in \Z}$ be the multiplier defined by
$\s_n = 1$ on $[-N, N]$,
$\s_n = \frac{N-j}N$ for $n = N+j$ and $n = -N -j$, $j = 1, \dots, N$,
and
$\s_n = 0$ for $|n| \geq 2N$.
Consider
\begin{equation}
K(x, t) :=  \sum_{n\in\Z} \s_n e^{2\pi i (nx + n^2t)}.
\label{D0a}
\end{equation}

\noi
Then, we have the following lemma.
Here, $\|x\| = \min_{n\in \Z}|x - n|$ denotes the distance
of $x$ to the closest integer as before.

\begin{lemma}[Lemma 3.18 in \cite{Bo2}] \label{LEM:Weyl}
Let $1 \leq a \leq q \leq N$ and $(a, q) = 1$
such that
\begin{equation}
\bigg\| t - \frac aq\bigg\| \leq \frac{1}{qN}.
\label{D0}
\end{equation}

\noi
Then, we have
\begin{equation}
| K(x, t)|
\les \frac{N}{q^\frac{1}{2}\Big(1 + N\big\|t - \frac aq\big\|^\frac{1}{2}\Big)}.
\label{D1}
\end{equation}

\end{lemma}

\noi
Note that
the multiplier $\s_n$ in \eqref{D0a} avoids
the logarithmic loss (when $q \sim N$) in Weyl's inequality \eqref{C2b}
on the Weyl sum $W_N(x, t) = \sum_{|n|\leq N} e^{2\pi i (nx + n^2t)}$.
Indeed, by writing $K = \frac{1}{N} \sum_{k = N}^{2N-1} W_k$,
we see that this regularizing effect in \eqref{D1} is analogous
to that of the F\'ej\'er kernel over the Dirichlet kernel.

In the following, we fix  $N\gg1$,  dyadic.
For dyadic $M \leq N$, let $\mathcal{R}_M$ by
\[ \mathcal{R}_M = \Big\{ \frac aq: \, a\in J_q, \, M\leq q < 2M\Big\},\]

\noi
where $J_q$ is as in \eqref{C1a}.
Let $\psi(t)$ be a smooth cutoff function
supported on $\frac{1}{2}+\frac{1}{10} \leq |t| \leq 2 - \frac{1}{10}$
such that $\sum_{j \in \Z} \psi(2^{-j} t) = 1$ for $t \ne 0$.
For $s \in \mathbb{N}$ with $ M \leq 2^s < N$
let $\omega_{N, 2^s} (t)= \psi(2^sN t)$
and define $\omega_{N, N}$ by
\[ \o_{N, N}(t) = \begin{cases}
\sum_{j \geq \log_2 N} \psi(2^jNt), & t\ne 0,\\
1, & t = 0.
\end{cases}
\]

\noi
Note that we have $\supp(\o_{N, 2^s}) \subset \big\{|t| \les \frac{1}{2^sN}\big\}$ and
\begin{equation}
|\ft \o_{N, 2^s}(k) | \les
\frac{1}{2^s N} \Big\langle\frac{ k}{2^sN}\Big\rangle^{-100}.
\label{D2a}
\end{equation}

\noi
Now, let
\begin{equation}
\Omega_{M, N} = \sum_{M \leq 2^s \leq N} \omega_{N, 2^s}.
\label{D2}
\end{equation}

\noi
Then,
it follows that $\O_{M, N} \equiv 1$ on $[ -\frac{1}{MN}, \frac{1}{MN}]$
and $\supp\O_{M, N} \subset [ -\frac{2}{MN}, \frac{2}{MN}]$.

Let $N_1 = \frac{1}{100}N$.
Note that, for $M_1 < M_2 \leq N_1$, we have
\[\bigg(\mathcal{R}_{M_1} + \Big[ -\frac{2}{M_1N}, \frac{2}{M_1N}\Big]\bigg)
\cap
\bigg(\mathcal{R}_{M_2} + \Big[ -\frac{2}{M_2N}, \frac{2}{M_2N}\Big] \bigg)= \emptyset.\]

\noi
Recall that by Dirichlet's theorem \cite[Lemma 2.1]{V},
\eqref{D0} is satisfied for all $t \in \T = [0, 1]$.
Then, by letting $\dl_T$ denote the Dirac delta measure at $T$, we have
\begin{equation}\label{D3}
1 = \sum_{\substack{M \leq N_1\\ M, \text{ dyadic}}} 
\sum_{T \in \mathcal{R}_M}\dl_T* \O_{M, N} + \rho,
\end{equation}

\noi
such that
$\rho(t) \ne 0$ for some $t \in \T$ implies that $t$ satisfies \eqref{D0}
for some $q > N_1$.
In particular, by Lemma \ref{LEM:Weyl}, we have
\begin{equation}
| \rho(t) K(x, t) | \les N^\frac{1}{2}.
\label{D4}
\end{equation}

\noi
From \eqref{D3} with \eqref{C4}, \eqref{D2a}, and \eqref{D2},
we have
\begin{align}
|\ft \rho(k)|
& \leq
\bigg| \sum_{\substack{M \leq N_1\\ M, \text{ dyadic}}} \sum_{M \leq 2^s \leq N}
\F \bigg[\sum_{T \in \mathcal{R}_M}\dl_T\bigg] (k)\cdot \ft \o_{N, 2^s}(k) \bigg| \notag \\
& \les \sum_{\substack{M \leq N_1\\ M, \text{ dyadic}}}
\sum_{M \leq 2^s \leq N}
\frac{d(k, M)M^{1+}}{2^s N} \Big\langle\frac{ k}{2^sN}\Big\rangle^{-100}
\les  N^{-1+}
\label{D4b}
\end{align}

\noi
for $k \ne 0$.
Here, we used the fact that $d(k, M) \les k^{0+}$
and $M \leq 2^s \leq N$.

Now, for each $M$ and $s$, we choose a coefficient $\al_{M, s}$ such that
\begin{equation}
\F\bigg[\sum_{T \in \mathcal{R}_M} \dl_T * \o_{N, 2^s}\bigg](0)
 =  \al_{M, s} \ft \rho(0).
\label{D4a}
 \end{equation}

\noi
Then, from  \cite[(3.56)]{Bo2}, we have
\begin{equation}
|\al_{M, s}| \les \frac{M^2}{2^sN}.
\label{D5}
\end{equation}

Now, we focus on our problem.
Namely, we do not assume $t \in \T$ any longer.
Given an interval $I \subset \R$,
assume that $I $ is centered at $0$
and let  $\chi$ be a smooth cutoff function support on $3I$
such that $\chi(t) \equiv 1$ on $2I$ as before.
We define
\begin{align}
\K(\bx, t) & = \chi(t) \sum_{\bn \in \wt S_N} \s_{n_1}
e^{2\pi i (\bn\cdot \bx + Q(\bn) t)} \notag \\
& = \chi(t) K(x_1, t) \prod_{j = 2}^d \sum_{|n_j|\leq N} e^{2\pi i (n_j x_j + \theta_j n_j^2 t)} ,
\label{D6}
\end{align}

\noi
where
$\wt S_N = \{ \bn \in \Z^d: \, |n_j| \leq N, j = 2, \dots, N\}$
and $K(x_1, t)$ is as in \eqref{D0a}.

Define $\Ld_{M, s}$ by
\begin{equation}
\Ld_{M, s}(\bx, t) = \K(\bx, t)
\bigg[\sum_{T \in \mathcal{R}_M} \dl_T * \o_{N, 2^s} (t) - \al_{M, s} \rho(t)\bigg] .
\label{D7}
\end{equation}

\noi
Then, from Lemma \ref{LEM:Weyl}, \eqref{D4}, and \eqref{D5} with $M \leq 2^s \leq N$, we have
\begin{align}
|\Ld_{M, s}(\bx, t)|\les N^{d-1} \frac{N}{M^\frac{1}{2} \big( 1 + (2^{-s} N)^\frac{1}{2}\big)}
+ \frac{M^2}{2^sN} N^{d-\frac{1}{2}}
\les N^{d-1} \Big(\frac{2^s N }{M}\Big)^\frac{1}{2}.
\label{D8}
\end{align}

\noi
Hence, from \eqref{D8}, we have
\begin{equation}
\| f * \Ld_{M, s} \|_{L^\infty(I \times \T^d)}
\les N^{d-1} \Big(\frac{2^s N }{M}\Big)^\frac{1}{2}\|f\|_{L^1(I \times \T^d)}.
\label{D9}
\end{equation}

Next, we estimate $|\ft \Ld_{M, s}|$.
Denote the second factor in \eqref{D7} by $\Phi_{M, s}$,
i.e.~let
\begin{equation}
\Phi_{M, s}(t) = \sum_{T \in \mathcal{R}_M} \dl_T * \o_{N, 2^s} (t) - \al_{M, s} \rho(t).
\label{D9a}
\end{equation}

\noi
Note that $\Phi_{M, s}$ is periodic.
Moreover, by  \eqref{D4a}, we have $\ft \Phi_{M, s}(0) = 0$.
Hence, we have
\begin{equation}
\ft \Ld_{M, s}(\bn, \tau) = \s_{n_1}
\bigg(\prod_{j = 2}^d \ind_{|n_j| \leq N}\bigg)
\sum_{k \ne 0}\ft  \Phi_{M, s}(k) \ft \chi(\tau - Q(\bn) - k).
\label{D10}
\end{equation}

%
%
\noi
By \eqref{C4}, \eqref{D2a}, \eqref{D4b},  and \eqref{D5}
with   $d(k, M) \les k^{0+}$
and $M \leq 2^s \leq N$,
 we have
\begin{align}
|\ft \Phi_{M, s}(k)|
\les \frac{d(k, M) M^{1+}}{2^s N} \Big\langle\frac{ k}{2^sN}\Big\rangle^{-100}
+ \frac{M^2}{2^s N^{2-}}
\les
\frac{M}{2^s N^{1-}}
\label{D11}
\end{align}

\noi
for $k \ne 0$.
By summing
$|\ft \chi(\tau - Q(\bn) - k)| \les \jb{\tau - Q(\bn) - k}^{-100}$
over $k \ne 0$,
it follows from \eqref{D10} and \eqref{D11} that
\begin{align}
|\ft \Ld_{M, s}(\bn, \tau)|
\les
\frac{M}{2^s N^{1-}}.
\label{D12}
\end{align}


\noi
Hence, from \eqref{D12}, we have
\begin{equation}
\| f * \Ld_{M, s} \|_{L^2(I \times \T^d)}
\les
\frac{M}{2^s N}N^{0+}
\|f\|_{L^2(I \times \T^d)}.
\label{D13}
\end{equation}

\noi
Also, with the trivial bound $d(k, M) \leq M$ in \eqref{D11}, we have
\begin{equation}
\| f * \Ld_{M, s} \|_{L^2(I \times \T^d)}
\les
\frac{M^{2+}}{2^s N}
\|f\|_{L^2(I \times \T^d)}.
\label{D13a}
\end{equation}

\noi
The second estimate \eqref{D13a} is useful when $M \ll N^\eps$.

\medskip

In the following, we establish another estimate  on
$\| f * \Ld_{M, s} \|_{L^2(I \times \T^d)}$,
using
the following lemma from Bourgain \cite{Bo2}.
\begin{lemma}[Lemma 3.47 in \cite{Bo2}] \label{LEM:divisor}
Let $d(k, M)$ denote the number of divisors of $k$ less than $M$.
Then, for any $\be, B, D > 0$, we have
\begin{equation}\label{D14}
\# \{ 0 \leq k \leq N: \, d(k, M) > D\} < c_{\be, B} (D^{-B} M^\be N + M^B).
\end{equation}

\noi
Note that the constant in \eqref{D14} is independent of $D>0$,
$M, N \in \mathbb{N}$.

\end{lemma}

\noi
From  \eqref{D10} and \eqref{D11}, we have
\begin{align}
\| f * \Ld_{M, s} \|_{L^2(I \times \T^d)}
& \les \frac{ M^{1+}}{2^s N}
\Bigg(\int
\sum_{\wt S_N} \s_{n_1}^2 |\ft f(\bn , \tau)|^2
\bigg[\sum_{k \ne 0}\frac{d(k, M)}{\jb{\frac{ k}{2^sN}}^{100}\jb{\tau - Q(\bn) - k}^{100}} \bigg]^2
 d\tau \Bigg)^\frac{1}{2}\notag \\
& \hphantom{XXX} + \frac{M^2}{2^s N^{2-}} \|f\|_{L^2(I\times \T^d)},
 \label{D15}
\end{align}

\noi
where $\wt S_N$ is as in \eqref{D6}.
Given $D > 0$ (to be chosen later),
separate the first term, depending on $d(k, M) \leq D$ or $> D$.
The contribution from $d(k, M) \leq D$ can be estimated by
\begin{equation}
 \les  \frac{ D M^{1+}}{2^s N} \|f\|_{L^2(I\times \T^d)}.
 \label{D16}
\end{equation}

\noi
Next, we estimate the contribution from $d(k, M) > D$.
By Cauchy-Schwarz inequality, we have
\begin{align}
\bigg[\sum_{k \ne 0} & \frac{d(k, M)}{\jb{\frac{ k}{2^sN}}^{100}\jb{\tau - Q(\bn) - k}^{100}} \bigg]^2
\notag \\
& \leq
\bigg(\sum_{k \ne 0}\frac{d(k, M)^2}{\jb{\frac{ k}{2^sN}}^{200}\jb{\tau - Q(\bn) - k}^{100}} \bigg)
\bigg(\sum_{\wt k \ne 0}\frac{1}{\jb{\tau - Q(\bn) - \wt k}^{100}} \bigg)
\notag \\
& \les
\bigg(\sum_{k \ne 0}\frac{d(k, M)^2}{\jb{\frac{ k}{2^sN}}^{200}\jb{\tau - Q(\bn) - k}^{100}} \bigg).
\label{D15a}
\end{align}

\noi
Now, we estimate the first term on the right-hand side of \eqref{D15}
after applying \eqref{D15a}.
By 
first integrating in $\tau$, then summing over
$|n_j| \les N$ for $j = 1, \dots, d$,
applying Lemma \ref{LEM:divisor}
(with $2B$ and $2\beta$ instead of $B$ and $\beta$),
the trivial bound  $d(k, M) \leq M$,
and Hausdorff-Young's inequality,
we have
\begin{align}
& \les
 \frac{ M^{1+}N^{\frac{d}{2}}}{2^s N}
\bigg(
\sum_{|k| \les 2^s N}
d(k, M)^2 +
\sum_{j = 1}^\infty
\sum_{|k| \sim 2^{s+j} N}d(k, M)^2\Big\langle\frac{ k}{2^sN}\Big\rangle^{-200}\bigg)^\frac{1}{2}
\|\ft f\|_{L^\infty_\tau \l^\infty_\bn}
\notag \\
& \les
 \frac{ M^{2+}N^{\frac{d}{2}}}{2^s N}  \bigg(
\sum_{j = 0}^\infty 2^{-200j} (D^{-2B} M^{2\be} 2^{s+j} N + M^{2B}) \bigg)^\frac{1}{2}
\|f\|_{L^1(I\times \T^d)}
\notag  \\
& \les
N^\frac{d}{2}\bigg(
\frac{D^{-B} M^{2+\be +} }{2^{\frac{s}{2}}N^\frac{1}{2}}
+  \frac{ M^{2+B+}}{2^sN } \bigg)
\|f\|_{L^1(I\times \T^d)}.
\label{D17}
\end{align}

\noi
Hence, from \eqref{D15}, \eqref{D16}, and \eqref{D17}
with $M \leq N$, we have
\begin{align}
\| f * \Ld_{M, s} \|_{L^2(I \times \T^d)}
& \les
\frac{ D M^{1+}}{2^s N} 
\|f\|_{L^2(I\times \T^d)} \notag \\
& \hphantom{XXX}
+
N^\frac{d}{2} \bigg(
\frac{D^{-B} M^{2+\be +} }{2^{\frac{s}{2}}N^\frac{1}{2}}
+  \frac{ M^{2+B+}}{2^sN } \bigg)
\|f\|_{L^1(I\times \T^d)}.
\label{D18}
\end{align}


Define $\Ld$ by
\begin{equation}
\Ld (\bx, t)
 = \sum_{\substack{M \leq N_1\\ M, \text{ dyadic}}}
\sum_{M \leq 2^s \leq N} \Ld_{M, s}(\bx, t),
\label{E1}
\end{equation}

\noi
where $\Ld_{M, s}$ is as in \eqref{D7}.
By \eqref{D3}, we have
\begin{equation}
\big(\K -  \Ld \big)(\bx, t)
= \bigg[ 1+
\sum_{\substack{M \leq N_1\\ M, \text{ dyadic}}}
\sum_{M \leq 2^s \leq N}\al_{M, s} \bigg] \K(\bx, t) \rho(t).
\label{E2}
\end{equation}

\noi
Then, by \eqref{D4}, \eqref{D5}, and \eqref{D6} with $N_1 = \frac{1}{100}N$, we have
\begin{equation}
\|\K - \Ld\|_{L^\infty(I\times\T^d)}
\les N^{d-\frac{1}{2}}.
\label{E3}
\end{equation}

\noi
Hence, we have
\begin{align}
|\jb{f, f*(\K - \Ld)}|
\les  N^{d-\frac{1}{2}} \|f\|_{L^1}^2. 
\label{E4}
\end{align}

Let $p \in (1, 2)$ such that
\begin{equation}
\frac 1p = \frac{1-\theta}{1}+\frac \theta 2
\label{E4a}
\end{equation}

\noi
for some $\theta \in (0, 1)$.
Note that $p' \theta = 2$.

\medskip

\noi
{\bf Case (i):}   $\theta < \frac{1}{5}$.

By interpolating \eqref{D9} and \eqref{D13a}, we have
\begin{align}
\| f * \Ld_{M, s} \|_{L^{p'}(I \times \T^d)}
\les N^{(d-1)(1-\theta)} (2^s N)^{\frac12 - \frac 32\theta}
M^{-\frac 12 + (\frac 52+)\theta}
\| f \|_{L^{p}(I \times \T^d)}.
\label{E4b}
\end{align}

\noi
Then, summing over dyadic $M \geq 1$ and $s$ with $2^s \leq N$,
we have
\begin{align}
\| f *\Ld  \|_{L^{p'}(I \times \T^d)}
& \les N^{d - (d+2)\theta}  
\| f \|_{L^{p}(I \times \T^d)},
\label{E4c}
\end{align}

\noi
as long as $\theta > 0$ satisfies $-\frac 12 + (\frac 52+)\theta \leq 0$,
i.e.
\begin{equation}
\theta < \frac 15.
\label{E4d}
\end{equation}

In the following, we prove  the level set estimate \eqref{E0}
for $ q > 10$.
Given $f$ as in \eqref{C0}, let $F(\bx, t) = e^{-it \Dl} f(\bx, t)$.
Let $\Theta_\ld(\bx, t)$ be as in \eqref{C5a},
where
$A_\ld = \{(\bx, t) \in \T^d \times I:\, |F(\bx, t)| > \ld\}$.
Then, proceeding as in \eqref{C5b}
with \eqref{E4}
and \eqref{E4c}, we have
\begin{align}
\ld^2|A_\ld|^2
& \leq \sum_{\bn \in S_N} \big|\ft \Theta_\ld(\bn, Q(\bn))\big|^2
\leq \jb{ \K* \Theta_\ld, \Theta_\ld}\notag \\
& \leq
|\jb{ (\K - \Ld) * \Theta_\ld, \Theta_\ld  }|+ |\jb{ \Ld*\Theta_\ld, \Theta_\ld }|\notag \\
& \les N^{d-\frac{1}{2}} |A_\ld|^2 
+  N^{d - (d+2) \theta} |A_\ld|^\frac{2}{p}.
\label{E4e}
\end{align}

\noi
For $\ld \gg N^{\frac{d}{2}-\frac 14}$,
\eqref{E4e} reduces to
\begin{align}
\ld^2|A_\ld|^2
& \les
 N^{d - (d+2) \theta} |A_\ld|^\frac{2}{p}.
\label{E4f}
\end{align}

\noi
Noting that $p' \theta = 2$ by \eqref{E4a},
it follows from \eqref{E4f} that
\begin{align*}
|A_\ld| \les
N^{\frac{d}{2}q - (d+2) } \ld^{-q},
\end{align*}

\noi
where $q: = p' = \frac 2\theta > 10$.
Note that we only needed to assume
$\ld \gg N^{\frac{d}{2}-\frac 14}$
and did not need the condition $\ld \ges N^{\frac d2 - \eps}$ in this case.

\medskip

\noi
{\bf Case (ii):} $\frac 15 \leq \theta <  \frac{1}{3}$.

Let $M_j$, $j = 1, 2$ be dyadic numbers such that
\begin{equation}
M_1 \sim \bigg(\frac{N^\frac{d}{2}}{\ld}\bigg)^{\dl_1}
\les N^{\eps \dl_1}
\quad \text{and} \quad M_2 \sim N^{\dl_2}.
\label{EE0}
\end{equation}

\noi
Here, we choose $\dl_1, \dl_2>0$
such that $M_1 \ll M_2$.
We divide $\Ld$ into three pieces:
$\Ld = \Ld_1 + \Ld_2 + \Ld_3$
by setting
\begin{equation}
\Ld_j = \sum_{\substack{M \in I_j \\M, \text{ dyadic}}}
\sum_{M \leq 2^s \leq N} \Ld_{M, s},
\label{EE0a}
\end{equation}

\noi
where
$I_1 = [1, M_1]$, $I_2 = (M_1, M_2]$, and $I_3 = (M_2, N_1]$
with $N_1 = \frac{1}{100}N$ as before.

Then, summing \eqref{E4b} over dyadic $M \leq M_1$
and $2^s \leq N$, we have
\begin{align}
\| f *\Ld_1  \|_{L^{p'}(I \times \T^d)}
& \les M_1^{-\frac{1}{2} + (\frac 52 +)\theta} N^{d - (d+2)\theta}
\| f \|_{L^{p}(I \times \T^d)},
\label{EE1}
\end{align}

\noi
since $\theta \in [\frac 15, \frac 13)$.

Similarly, by interpolating \eqref{D9} and \eqref{D18},
we have
\begin{align}
\| f * & \Ld_{M, s} \|_{L^{p'}(I \times \T^d)}
 \les N^{(d-1)(1-\theta)} (2^s N)^{\frac12 - \frac 32\theta}
M^{-\frac{1}{2}+ (\frac{3}{2}+)\theta}
D^\theta  \| f \|_{L^{p}(I \times \T^d)}\notag \\
& 
+ N^{(d-1)(1-\theta)} \bigg(\frac{2^s N}{M}\bigg)^{\frac12 (1- \theta)}
N^{\frac{d}{2} \theta}
\bigg(\frac{D^{-B}M^{2+\be+}}{2^\frac{s}{2} N^\frac{1}{2}}
+  \frac{ M^{2+B+}}{2^sN } \bigg)^\theta
\|f\|_{L^1(I\times \T^d)}.
\label{E5}
\end{align}

\noi
Now, choose $D\sim M^\al$ for some small $\al > 0$.
Then, set
$\be \ll 1$
and $ B \gg 1$ such that
\begin{equation}
\s: =- \frac{5}{2} - \beta + \al B - > 0.
\label{E5a}
\end{equation}

\noi
Then, summing \eqref{E5} over dyadic $M \in (M_1, M_2]$ and $s$
with $2^s \leq N$, we have
\begin{align}
\| f * \Ld_2   \|_{L^{p'}(I \times \T^d)}
&  \les
N^{d - (d+2) \theta}
 \| f \|_{L^{p}(I \times \T^d)} \notag \\
&
+ \Big( M_1^{-\frac{1}{2} - \s \theta} N^{d - (\frac d 2+1)\theta}
+
M_2^{-\frac{1}{2} + (\frac{5}{2}+  B+) \theta} N^{d - (\frac d 2+2) \theta} \Big)
\|f\|_{L^1(I\times \T^d)},
\label{E6}
\end{align}

\noi
as long as $-\frac{1}{2} + (\frac{3}{2}+\al+)\theta \leq 0$,
i.e.~
\begin{equation}
\theta < \frac{1}{3+2\al +}.
\label{E6a}
\end{equation}

\noi
Note that \eqref{E6a} can be satisfied
as long as $\theta < \frac{1}{3}$
by choosing $\al$ sufficiently small.

%

Lastly, from \eqref{D9} and \eqref{D13}, we have
\begin{align}
\| f * \Ld_{M, s} \|_{L^{p'}(I \times \T^d)}
\les N^{(d-1)(1-\theta)} \bigg(\frac{2^s N}{M}\bigg)^{\frac12 - \frac 32\theta}
 N^{(0+) \theta}
\| f \|_{L^{p}(I \times \T^d)}.
\label{E7}
\end{align}

\noi
Then, summing over $M \geq M_2 \sim N^{\dl_2}$ and $s$ with $2^s \leq N$,
we have
\begin{align}
\| f * \Ld_3   \|_{L^{p'}(I \times \T^d)}
& \les N^{(d-1)(1-\theta)} \bigg(\frac{N^2}{M_2}\bigg)^{\frac12 - \frac 32\theta} N^{(0+) \theta}
\| f \|_{L^{p}(I \times \T^d)}\notag \\
& \les N^{d - (d+2)\theta}  
\| f \|_{L^{p}(I \times \T^d)},
\label{E8}
\end{align}

\noi
as long as $\theta < \frac 13$.

Putting \eqref{EE1}, \eqref{E6},  and \eqref{E8} together
with \eqref{E4d}, we obtain
\begin{align}
\| f * \Ld \|_{L^{p'}(I \times \T^d)}
& \les  M_1^{-\frac{1}{2} + (\frac 52 +)\theta}
N^{d - (d+2) \theta}
 \| f \|_{L^{p}(I \times \T^d)} \notag \\
&
+ \Big(   M_1^{- \frac 12 -\s \theta} N^{d - (\frac d 2+1)\theta}
+
M_2^{ -\frac{1}{2}+(\frac{5}{2}+B+) \theta} N^{d - (\frac d 2+2) \theta} \Big)
\|f\|_{L^1(I\times \T^d)}.
\label{E9}
\end{align}

\medskip
Now, we are ready to prove the level set estimate \eqref{E0} for $ q > 6$.
Then, proceeding as in Case (i)
with \eqref{E4}
and \eqref{E9}, we have
\begin{align}
\ld^2|A_\ld|^2
& \leq
|\jb{ (\K - \Ld) * \Theta_\ld, \Theta_\ld  }|+ |\jb{ \Ld*\Theta_\ld, \Theta_\ld }|\notag \\
& \les N^{d-\frac{1}{2}} |A_\ld|^2
+  M_1^{-\frac{1}{2} + (\frac 52 +)\theta} N^{d - (d+2) \theta} |A_\ld|^\frac{2}{p} \notag \\
& \hphantom{XXX}
+ M_1^{- \frac 12- \s \theta} N^{d - (\frac{d}{2}+1) \theta}|A_\ld|^{1+\frac{1}{p}}
+ M_2^{ -\frac{1}{2}+(\frac{5}{2}+B+) \theta} N^{d - (\frac d 2+2) \theta}
|A_\ld|^{1+\frac{1}{p}}.
\label{E10}
\end{align}

\noi
Since $\ld \gg N^{\frac{d}{2}-\frac 14}$,
\eqref{E10} reduces to
\begin{align}
\ld^2|A_\ld|^2
& \les
  M_1^{-\frac{1}{2} + (\frac 52 +)\theta} N^{d - (d+2) \theta} |A_\ld|^\frac{2}{p}
+ M_1^{- \frac 12- \s \theta} N^{d - (\frac{d}{2}+1) \theta}|A_\ld|^{1+\frac{1}{p}}\notag \\
& \hphantom{XXX}
+ M_2^{ -\frac{1}{2}+(\frac{5}{2}+B+) \theta} N^{d - (\frac d 2+2) \theta}
|A_\ld|^{1+\frac{1}{p}} \notag \\
%
& = : \I + \II + \III.
\label{E11}
\end{align}

First, suppose that $\ld^2|A_\ld|^2 \les \I$ holds.
Recall from \eqref{E4a} that $p' \theta = 2$.
Then, with
\eqref{EE0},
we have
\begin{align}
|A_\ld| \les
\bigg(\frac{N^\frac{d}{2}}{\ld}\bigg)^{(-\frac{p'}{4}+\frac{5}{2}+)\dl_1}
N^{\frac{d}{2}p' - (d+2) } \ld^{-p'}
\les
N^{\frac{d}{2}q - (d+2) } \ld^{-q}
\label{E12}
\end{align}

\noi
for $ q > p'$ by choosing $\dl_1 = \dl_1(q, p')$ sufficiently small.

Next, suppose that
$\ld^2|A_\ld|^2 \les \II$ holds.
Then,  from  \eqref{EE0},
we have
\begin{align}
|A_\ld|
& \les
N^{\frac{d}{2}p' - (d+2) } \ld^{-p'}
\Big(N^{\frac{d}{2}p'} \ld^{-p'}M_1^{-\frac{p'}{2} - 2\s}\Big)
 \les N^{\frac d 2 p' - (d+2)} \ld^{-p'},
\label{E13}
\end{align}

\noi
by making $\s = \s(p', \dl_1)$ in \eqref{E5a}
(and hence $B = B(p', \dl_1)$)
sufficiently large.

Lastly, suppose that
$\ld^2|A_\ld|^2 \les \III$ holds.
By $\ld \ges N^{\frac{d}{2}-\eps}$ and  \eqref{EE0},
 we have
\begin{align}
|A_\ld|
& \les
 N^{\frac d 2 p' - (d+2)} \ld^{-p'}
N^{-2 + \eps p' +(-\frac{p'}{2} + 5+ 2B+)\dl_2}
\les  N^{\frac d 2 p' - (d+2)} \ld^{-p'}
\label{E15}
\end{align}

\noi
as long as
we have $\eps p' \leq 1$
and
$\dl_2 = \dl_2(p', B)$ is sufficiently small
such that $(-\frac{p'}{2} + 5+ 2B+)\dl_2 \leq 1$.

Finally, given $q > 6$,
we choose $\theta < \frac 13$ such that
$ q > p' = \frac{2}{\theta} > 6$.
Then,  from \eqref{E11}, \eqref{E12}, \eqref{E13}, and \eqref{E15}
with $\ld \leq N^\frac{d}{2}$,
we obtain
\begin{align*}
|A_\ld|
 \les N^{\frac d 2 q - (d+2)} \ld^{-q}.
\end{align*}

\noi
This completes the proof of Proposition \ref{PROP:level} (ii).

%
%
%
%
%
%
%
%
%
%

%

%
%
\section{Well-posedness in subcritical spaces}
\label{SEC:4}

In this section, we prove local well-posedness of NLS \eqref{eq:nls1}
on irrational tori
in subcritical Sobolev spaces (Theorem \ref{THM:3}).
%
%
It turns out that the well-posed theory of \eqref{eq:nls1} is very similar
to that on the standard tours \cite{Bo2,Bo1, BoPCMI, Bo4}.

In the seminal paper \cite{Bo2}, Bourgain
introduced the $X^{s, b}$-space
whose norm is given by
\begin{equation}
\| u \|_{X^{s, b}(\R \times \T^d)} = \| \jb{\bn}^s \jb{\tau - |\bn|^2}^b\ft u(\bn, \tau)
\|_{L^2_\tau\ell^2_\bn(\R\times \mathbb{Z}^d )},
\label{F1}
\end{equation}

\noi
where $\jb{\, \cdot\, } = (1+|\cdot|^2)^\frac{1}{2}$.
After establishing Strichartz estimates,
he proved several well-posedness
results of NLS on the standard torus $\T^d$.
In our case, i.e.~on an irrational torus,
we need to replace the weight $\jb{\tau - |\bn|^2}^b$ in \eqref{F1}
by $\jb{\tau - Q(\bn)}^b$,
where $Q(\bn)$ is defined in \eqref{eq:Q}.
Then,
by the standard $X^{s,b}$-theory,
it is known that certain multilinear Strichartz estimates imply well-posedness.
More precisely, we have the following lemma.

\begin{lemma}\label{LEM:LWP1}
Let $s_0 > \max (0, s_c)$
and $I\subset \R$ be a bounded  interval.
Suppose that the following multilinear Strichartz estimate holds  for $s>s_0$:
\begin{align}\label{multi-est}
\bigg\|\prod_{j=1}^{k+1}e^{-it\Delta}\phi_j \bigg\|_{L^2_{t,\bx}(I\times \T^d)}
\les N_{\max}^{-s}
\prod_{j=1}^{k+1} N_j^{s}\|\phi_i\|_{L^2_\bx(\T^d) },
\end{align}

\noi
for all $\phi_j \in L^2(\T^d)$
with $\supp \widehat \phi_j \subset [-N_j,N_j]^d$, $j = 1, \dots,  k+1$,
and  $N_{\max} := \max(N_1,\cdots,N_{k+1} )$.
Then,
the Cauchy problem \eqref{eq:nls1} is locally well-posed in $H^s(\T^d)$ for $s> s_0$.
\end{lemma}

\noi
The proof of Lemma \ref{LEM:LWP1} is standard
and
we refer the readers to \cite{Bo2,Bo1,Bo4,CW} for details.

\begin{proof}[Proof of Theorem \ref{THM:3}]
In view of Lemma \ref{LEM:LWP1},
it suffices to prove 
the $ (k+1)$-linear estimate \eqref{multi-est} for $ s > s_0$.
Without loss of generality, assume that $N_1 \ge N_2\ge \cdots \ge N_{k+1}$.


\medskip

\noi
{\bf Case (i):}
 $d=2$. When $k=1$, the well-posedness result was already obtain in \cite{CW}.
In the following, we first consider the case $k =  2, 3, 4, 5$.
First, assume that $\supp \ft \phi_1 \subset [-N_2, N_2]^2$.
Then,
 by H\"older's inequality and \eqref{co-d2} with $\frac{4(7k+5) }{3(k+3)} \in ( 4, \frac{20}{3}]$, we have
\begin{align}
\bigg\|\prod_{j=1}^{k+1} e^{-it\Delta} \phi_j \bigg\|_{L^2_{t,\bx}}
& \le  \|e^{-it\Delta} \phi_1 \|_{L^\frac{4(7k+5) }{3(k+3)}} \|e^{-it\Delta} \phi_2 \|_{L^\frac{4(7k+5) }{3(k+3)}}
\prod_{j=3}^{k+1} \|e^{-it\Delta} \phi_j \|_{L^\frac{7k + 5}{2}} \notag \\
& \les  N_2^{\frac{7k-3}{7k+5} + 2\varepsilon}
\|\phi_1 \|_{L^2}\|\phi_2 \|_{L^2}\prod_{j=3}^{k+1}
 N^{\frac{7k-3}{7k+5}+\eps}_j \|\phi_j \|_{L^2}\notag \\
& \leq N_{\max}^{-s}
\prod_{j=1}^{k+1} N_j^{s}\|\phi_i\|_{L^2_\bx},
\label{F2}
\end{align}

\noi
for
$s > s_0 = \frac{7k-3}{7k+5}$.
%
%
%
%
In general, if $\supp \ft \phi_1 \subset [-N_1, N_1]^2$,
then we can write $\phi_1 = \sum_{|\bj|\les \frac{N_1}{N_2}} \phi_{1\bj}$,
where $\supp \ft \phi_{1\bj} \subset  N_2\,  \bj + [-N_2, N_2]^2$.
Letting $\psi_{1\bj} (\bx) = e^{ - 2\pi i  N_2 \bj \cdot \bx} \phi_{1\bj}(\bx) $,
we have  $\supp \ft \psi_{1\bj} \subset [-N_2, N_2]^2$.
Then, by a change of variables and \eqref{co-d2} (see \cite{BoPCMI}), we obtain
\begin{equation}
 \|e^{-it\Delta} \phi_{1\bj} \|_{L^\frac{4(7k+5) }{3(k+3)}_{t, \bx}}
=   \|e^{-it\Delta} \psi_{1\bj} \|_{L^\frac{4(7k+5) }{3(k+3)}}
 \les N_2^{\frac{7k-3}{2(7k+5)} + \eps}\|\phi_{1\bj}\|_{L^2_\bx}.
\label{F3}
\end{equation}

\noi
Then, by almost orthogonality with \eqref{F2} and \eqref{F3},
we have
\begin{align*}
\bigg\|\prod_{j=1}^{k+1}e^{-it\Delta}\phi_j \bigg\|_{L^2_{t,\bx}}^2
\les \sum_{|\bj|\les \frac{N_1}{N_2}}
\bigg\|e^{-it\Delta}\phi_{1\bj}  \prod_{j=2}^{k+1}e^{it\Delta}\phi_j \bigg\|_{L^2_{t,\bx}}^2
\les N_{\max}^{-2s} \prod_{j=1}^{k+1} N_j^{2s}\|\phi_j\|_{L^2_\bx}^2.
\end{align*}

\begin{remark}\rm
In view of the reduction above,
we only prove \eqref{multi-est}, assuming  $\supp \ft \phi_1 \subset [-N_2, N_2]^d$
in the following.
\end{remark}

Next, we consider the case $k \geq 5$.
By H\"older's inequality and \eqref{co-d2} with
$\frac{8k}{k+1} \in [ \frac{20}{3}, 10)$, we have
\begin{align*}
\bigg\|\prod_{j=1}^{k+1} e^{-it\Delta} \phi_j \bigg\|_{L^2_{t,\bx}}
& \le  \|e^{-it\Delta} \phi_1 \|_{L^\frac{8k}{k+1}}
\|e^{-it\Delta} \phi_2 \|_{L^\frac{8k}{k+1}}
\prod_{j=3}^{k+1} \|e^{-it\Delta} \phi_j \|_{L^4k} \notag \\
& \les  N_2^{1- \frac{1}{k} + 2\varepsilon}
\|\phi_1 \|_{L^2}\|\phi_2 \|_{L^2}\prod_{j=3}^{k+1}
 N^{1- \frac{1}{k}}_j \|\phi_j \|_{L^2_\bx}.
\end{align*}

\noi
Hence, \eqref{multi-est} holds for $s > s_0 = 1 - \frac{1}{k}$.

\medskip

\noi
{\bf Case (ii):}
$d=3$.
When $k=1$, the well-posedness result was obtained in \cite{Bo4}.
When $k=2$, by H\"older's inequality and \eqref{co-d3}, we have
\begin{align*}
\bigg\|\prod_{j=1}^{3} e^{-it\Delta} \phi_j \bigg\|_{L^2_{t,\bx}}
& \le  \|e^{-it\Delta} \phi_1 \|_{L^{\frac{104}{21}}} \|e^{-it\Delta} \phi_2 \|_{L^{\frac{104}{21}}}
\|e^{-it\Delta} \phi_3 \|_{L^{\frac{52}5}} \\
& \les  N_2^{\frac{53}{52} + 2 \eps }
 N^{\frac{53}{52}}_3 \prod_{j = 1}^3 \|\phi_j \|_{L^2_\bx}.
\end{align*}

\noi
\noi
Hence, \eqref{multi-est} holds for $s > s_0 = \frac{53}{52}$.
When $k\ge 3$, by H\"older's inequality and \eqref{co-d3}, we have
\begin{align*}
\bigg\|\prod_{j=1}^{k+1} e^{-it\Delta} \phi_j \bigg\|_{L^2_{t,\bx}}
& \le  \|e^{-it\Delta} \phi_1 \|_{L^{\frac{20k}{3k+2}}} \|e^{-it\Delta} \phi_2 \|_{L^{\frac{20k}{3k+2}}} \prod_{j=3}^{k+1} \|e^{-it\Delta} \phi_j \|_{L^{5k}} \\
& \les  N_2^{\frac32- \frac1{k}} \|\phi_1 \|_{L^2} \|\phi_2 \|_{L^2} \prod_{j=3}^{k+1} N^{\frac32- \frac1{k}}_j \|\phi_j \|_{L^2_\bx}.
\end{align*}

\noi
Hence,  \eqref{multi-est} holds $s \geq s_0 = \frac32- \frac1{k}$.

\medskip

\noi
{\bf Case (iii):}
 $d\ge 4$ and $k\ge 1$.
 By H\"older's inequality with  \eqref{co-d4} or \eqref{co-d5}, we have
\begin{align*}
\bigg\|\prod_{j=1}^{k+1} e^{-it\Delta} \phi_j \bigg\|_{L^2_{t,\bx}}
& \le  \|e^{-it\Delta} \phi_1 \|_{L^{\frac{4(d+2)k}{dk+2}}} \|e^{-it\Delta} \phi_2 \|_{L^{\frac{4(d+2)k}{dk+2}}} \prod_{j=3}^{k+1} \|e^{-it\Delta} \phi_j \|_{L^{(d+2)k}} \\
& \les  N_2^{\frac d2- \frac1{k}+\varepsilon} \|\phi_1 \|_{L^2} \|\phi_2 \|_{L^2}\prod_{j=3}^{k+1} N^{\frac d2- \frac1{k}}_j \|\phi_j \|_{L^2_\bx}.
\end{align*}

\noi
Hence,  \eqref{multi-est} holds  for $s > s_0 = \frac d2- \frac1{k}$.
\end{proof}

\section{Well-posedness in critical spaces}
\label{SEC:critical}

\subsection{Function spaces} \label{SUBSEC:G1}

In this section, we prove local well-posedness of NLS \eqref{eq:nls1}
on irrational tori
in critical Sobolev spaces  $H^{s_c}(\T^d)$ (Theorem \ref{THM:4}).
In the following, we use
the $U^p$- and $V^p$-spaces, developed by
Tataru, Koch, and their collaborators \cite{KochT, HHK, HTT11, HTT2}.
These spaces have been very effective in
establishing well-posedness  of various dispersive PDEs in critical spaces.
We briefly go over the basic definitions
of function spaces and   their properties.
See Hadac-Herr-Koch \cite{HHK} and Herr-Tataru-Tzvetkov \cite{HTT11} for detailed proofs.

Let $H$ be a separable Hilbert space over $\C$.
 In particular, it will be either $H^s(\T^d)$ or $\C$.
%
 Let $\mathcal{Z}$ be the collection of finite partitions $\{t_k\}_{k = 0}^K$ of $\R$:
 $-\infty < t_0 < \cdots < t_K \leq \infty$.
 If $t_K = \infty$,
 we use the convention $u(t_K) :=0$
 for all functions $u:\R\to H$.
We use $\ind_I$ to denote the sharp characteristic function
of a set $I \subset \R$.

 \begin{definition} \label{DEF:X1}\rm
Let $1\leq p < \infty$.

\smallskip
\noi
\textup{(i)}
A $U^p$-atom is defined by a step function $a:\R\to H$
of the form
\[ a = \sum_{k = 1}^K \ind_{[t_{k-1}, t_k)}\phi_{k - 1}, \]

\noi
where $\{t_k\}_{k = 0}^K \in \mathcal{Z}$
and $\{\phi_k\}_{k = 0}^{K-1} \subset H$
with $\sum_{k = 0}^{K-1} \|\phi_k\|_H^p = 1$.
Then, we define the atomic space $U^p(\R; H)$
to be the collection of functions $u:\R\to H$
of the form
\begin{equation} u = \sum_{j = 1}^\infty \ld_j a_j,
\quad \text{ where $a_j$'s are $U^p$-atoms and $\{\ld_j\}_{j \in \mathbb{N}}\in \l^1(\mathbb{N}; \C)$},
\label{X1}
\end{equation}

\noi
with the norm
\[ \|u\|_{U^p(\R; H)} : = \inf \Big\{ \|{ \bf \ld} \|_{\l^1}
: \eqref{X1} \text{ holds with } \ld = \{\ld_j \}_{j \in \mathbb{N}}
\text{ and some $U^p$-atoms } a_j\Big\}.\]

\smallskip
\noi
\textup{(ii)}
We define $V^p(\R; H)$
by the collection of functions $u : \R \to H$
with $\|u\|_{V^p(\R; H)} < \infty$,
where the $V^p$-norm is defined by
\[
\|u\|_{V^p(\R; H)}
:= \sup_{\{t_k\}_{k = 0}^K \in \mathcal{Z}}
\bigg(\sum_{k = 1}^K\|u(t_k) - u(t_{k-1})\|_H^p\bigg)^\frac{1}{p}.
\]

\noi
We also define $V^p_\text{rc}(\R; H)$
to be the closed subspace of all right-continuous functions
in $V^p(\R; H)$ such that
$\lim_{t \to -\infty} u(t) = 0$.

\smallskip
\noi
\textup{(iii)}
Let $s \in \R$.
We define $U^p_\Dl H^s$ (and $V^p_\Dl H^s$, respectively)
to be the spaces of all functions $u: \R \to H^s(\T^d)$
such that the following
$U^p_\Dl H^s$-norm (and $V^p_\Dl H^s$-norm, respectively)
is finite:
\[ \|u \|_{U^p_\Dl H^s} := \|e^{it \Dl} u\|_{U^p(\R; H^s)}
\quad \text{and} \quad
\|u \|_{V^p_\Dl H^s} := \|e^{it \Dl} u\|_{V^p(\R; H^s)}.\]

\noi
Here, the Laplacian $\Dl$ is defined
in terms of $Q(\bn)$ as in \eqref{eq:Q0}.
 \end{definition}

\begin{remark}\label{REM:UpVp} \rm
Note that
the spaces $U^p(\R; H)$,
$V^p(\R; H)$,  and $V^p_\text{rc}(\R; H)$
are Banach spaces.
Moreover, we have the following embeddings:
\begin{equation*}
U^p(\R; H) \hookrightarrow V^p_\text{rc}(\R; H)
\hookrightarrow U^q(\R; H)   \hookrightarrow L^\infty(\R; H)
\end{equation*}

\noi
for $ 1\leq p < q < \infty$.
Similar embeddings hold for $U^p_\Dl H^s$ and $V^p_\Dl H^s$.
%
%
\end{remark}

Next, we state a transference principle and an interpolation result.

\begin{lemma}\label{LEM:Xinterpolate}
\textup{(i)}
{\rm (transference principle)}
Suppose that we have
\[ \big\| T(e^{-it \Dl} \phi_1, \dots, e^{-it \Dl} \phi_k)\big\|_{L^p_t L^q_\bx(\R\times \T^d)}
\les \prod_{j = 1}^k \|\phi_j\|_{L^2_\bx}\]

\noi
for some $1\leq p, q \leq \infty$.
Then, we have
\[ \big\| T(u_1,  \dots, u_k)\big\|_{L^p_t L^q_\bx(\R\times \T^d)}
\les \prod_{j = 1}^k \|u_j\|_{U^p_\Dl L^2_\bx}.\]

\noi
\textup{(ii)}
{\rm (interpolation)}
 Let $E$ be a Banach space.
Suppose that $T: U^{p_1}\times \cdots \times U^{p_k} \to E$
is a bounded $k$-linear operator such that
\[ \|T(u_1, \dots, u_k)\|_{E} \leq C_1 \prod_{j = 1}^k \|u_j\|_{U^{p_j}}\]

\noi
for some $p_1, \dots, p_k > 2$.
Moreover, assume that there exists $C_2 \in (0, C_1]$
such that
\[ \|T(u_1, \dots, u_k)\|_{E} \leq C_2 \prod_{j = 1}^k \|u_j\|_{U^{2}}.\]

\noi
Then, we have
\[ \|T(u_1, \dots, u_k)\|_{E} \leq C_2 \Big(\ln \frac{C_1}{C_2}+ 1\Big)^k
 \prod_{j = 1}^k \|u_j\|_{V^2}\]

\noi
for $u_j \in V^2_\textup{rc}$, $j = 1, \dots, k$.
\end{lemma}

\noi
A transference principle as above has been commonly used
in the Fourier restriction norm method.
 See \cite[Proposition 2.19]{HHK}
for the proof of  Lemma \ref{LEM:Xinterpolate} (i).
The proof of the interpolation result follows from extending the trilinear result in \cite{HTT11}
to a general $k$-linear case.
See also   \cite[Proposition 2.20]{HHK}.

\smallskip

Let $\eta: \R \to [0, 1]$ be an even smooth cutoff function
supported on $[-\frac{8}{5}, \frac{8}{5}]$
such that $\eta \equiv 1$ on $[-\frac{5}{4}, \frac{5}{4}]$.
Given a dyadic number $N \geq 1$, we
set
$\eta_1(\xi) = \eta(|\xi|)$
and
\[\eta_N(\xi) = \eta\Big(\frac{|\xi|}{N}\Big) - \eta\Big(\frac{2|\xi|}{N}\Big)\]

\noi
for $N \geq 2$.
Then, we define the Littlewood-Paley projection operator
$P_N$ as the Fourier multiplier operator with symbol $\eta_N$.
Moreover, we define $P_{\leq N}$
by $P_{\leq N} = \sum_{1 \leq M \leq N} P_M$.
More generally,
given a set $R \subset \Z^d$,
we define $P_R$
to be the Fourier multiplier operator
with symbol $\ind_R$.

\begin{definition} \label{DEF:X3}
\textup{(i)}
Let $s\in \R$.
We define $X^s$ to be the space of all functions
$u : \R \to H^s(\T^d)$
such that $\| u\|_{X^s} < \infty$,
where the $X^s$-norm is defined by
\[ \|u \|_{X^s} : = \bigg( \sum_{\bn \in \Z^d} \jb{\bn}^{2s}
\big\| e^{- i t Q(\bn)} \ft{u}(\bn, t) \big\|_{U^2(\R_t; \C)}^2\bigg)^\frac{1}{2}.
\]

\smallskip
\noi
\textup{(ii)}
Let $s\in \R$.
We define $Y^s$ to be the space of all functions
$u : \R \to H^s(\T^d)$
such that
for every $\bn \in \Z^d$, the map $t \mapsto e^{- it Q(\bn)} \ft{u}(\bn, t)$
is in $V^2_\textup{rc}(\R_t; \C)$
and
$\| u\|_{Y^s} < \infty$,
where the $Y^s$-norm is defined by
\[ \|u \|_{Y^s} : = \bigg( \sum_{\bn \in \Z^d} \jb{\bn}^{2s}
\big\| e^{- i t Q(\bn)} \ft{u}(\bn, t) \big\|_{V^2(\R_t; \C)}^2\bigg)^\frac{1}{2}.
\]

\end{definition}

\noi
Recall the following embeddings:
\begin{equation}
U^2_\Dl H^s \hookrightarrow X^s \hookrightarrow
Y^s \hookrightarrow V^2_\Dl H^s.
\label{X3}
\end{equation}

\noi
Given
a time interval $I \subset \R$,
we define the restrictions $X^s(I)$
and $Y^s(I)$ of these spaces
in the usual manner.

We now state the linear estimates.
Given $f \in L^1_\text{loc}([0, \infty); L^2(\T^d))$,
define $\mathcal{I}(f)$ by
\[ \mathcal{I}(f)(t): = \int_0^t e^{-i(t- t') \Dl} f(t') dt'.\]

 \begin{lemma}[Linear estimates]\label{LEM:Xlinear}
 Let $s   \geq 0$ and $T > 0$.
 Then, the following linear estimates hold:
 \begin{align*}
  \|e^{-it \Dl} \phi\|_{X^s([0, T))}
  & \le \|\phi\|_{H^s}, \\
  \|\mathcal{I}(f)\|_{X^s([0, T))}
 & \leq \sup_{\substack{v \in Y^{-s}([0, T))\\\|v\|_{Y^{-s} = 1}}}
 \bigg| \int_0^T \int_{\T^d} f(\bx, t) \cj{v(\bx, t)} d\bx dt\bigg|,
 \end{align*}

 \noi
for all $\phi \in H^s(\T^d)$ and $f \in L^1([0, T); H^s(\T^d))$.
 \end{lemma}

Next, we present the crucial multilinear estimate.

\begin{proposition}\label{PROP:XLWP}
Let $d$ and $k$ satisfy
  \begin{align}\label{admis1}
\textup{(i) } d = 2, \ \, k \ge 6, \quad
    \textup{(ii) }d= 3, \ \, k\ge3,
    \quad \text{ or } \quad
    \textup{(iii) }    d\ge 4, \ \, k\ge 2.
  \end{align}

\noi
Then,  the following multilinear estimate holds
for all  $T \in (0, 1]$:
  \begin{align}
     \bigg\| \mathcal{I}
  \bigg(   \prod_{j=1}^{2k+1}u^*_j\bigg)\bigg\|_{X^{s_c}([0, T))}
     \les
     \sum_{j = 1}^{2k+1} \bigg(\|u_j\|_{X^s([0, T))}
     \prod_{\substack{ \l = 1\\ \l \ne j}}^{2k+1} \|u_\l\|_{X^{s_c}([0, T))}\bigg),
\label{X0}
  \end{align}

\noi
for $s \geq s_c  = \frac d2 - \frac 1k >0$,
where $u^*_j$ denotes either $u_j$ or $\cj u_j$.
\end{proposition}

\noi
Once we prove Proposition \ref{PROP:XLWP},
one can prove
Theorem \ref{THM:4}
by the fixed point argument as in \cite{HTT11,W}.
We omit details.
The remainder of this section is devoted to the proof of Proposition \ref{PROP:XLWP}.
Indeed, the multilinear estimate \eqref{X0}
follows once we prove the following
 multilinear Strichartz estimate.

\begin{proposition}\label{PROP:LWP2}
Let $d$ and $k$ satisfy \eqref{admis1}.
Then,  there exists $\dl > 0$ such that the following multilinear Strichartz estimate holds:
  \begin{align}
     \bigg\|\prod_{j=1}^{k+1}P_{N_j} e^{-it \Dl} \phi_j\bigg\|_{L^2_{t, \bx}(I \times\T^d)}
     \les& \left(\frac{N_{k+1}}{N_1}+\frac1{N_2}\right)^{\dl}\|P_{N_1}\phi_1\|_{L^2_\bx(\T^d)}
     \prod_{j=2}^{k+1}N^{s_c}_j \|P_{N_j}\phi_j\|_{L^2_\bx(\T^d)},
\label{G1}
  \end{align}

\noi
for
any interval $I\subset [0, 1]$
and for
all $\phi_j \in L^2(\T^d)$, $j = 1, \dots, k+1$,
and $N_1\ge N_2 \ge \cdots \ge N_{k+1}\ge 1$.
\end{proposition}

\noi

In Subsection \ref{SUBSEC:G2},
we first present the proof of
Proposition \ref{PROP:LWP2}.
In Subsection \ref{SUBSEC:G3},
we then
use the multilinear Strichartz estimate \eqref{G1}
to prove Proposition \ref{PROP:XLWP},
thus yielding Theorem \ref{THM:4}.

\subsection{Multilinear Strichartz estimate} \label{SUBSEC:G2}
In this subsection,
we use
 the sharp $L^p$-Strichartz estimates \eqref{P1} in Theorem \ref{THM:2}
to prove the multilinear Strichartz estimate \eqref{G1}.
The main idea is to
refine the Strichartz estimate
by considering frequency scales
finer than the standard dyadic Littlewood-Paley localizations
as in   \cite{HTT11}.
See Lemma \ref{LEM:Gstr2}.

\begin{definition}\label{DEF:Gadmis} \rm
We say that $(d, p)\in \mathbb{N}\times \R$ is an admissible pair if
  \begin{align}\label{admis}
\textup{(i) } d = 2, \ \, p > \frac{20}{3}, \quad
\textup{(ii) }    d =3, \  \, p>\frac{16}3,  \quad \text{ or}
\quad   \textup{(iii) }  d  \ge 4,\ \, p> 4.
  \end{align}
\end{definition}

\noi
Note that, by Theorem \ref{THM:2},
the Strichartz estimates with $(d,p)$ in this range
are  sharp.

\medskip
Given dyadic $N\geq 1$,
let
 $\mathcal{C}_N$ denote the collection of cubes $C\subset \Z^d$ of side
 length $\sim N$ with arbitrary center and orientation.
 Then,  we can rewrite Theorem \ref{THM:2} in the following form.

\begin{lemma}\label{LEM:Gstr}
Let $(d, p)$ be admissible
and $I\subset \R$ be a bounded interval.
Then, for all dyadic $N\geq 1$, we have
  \begin{align}\label{G2}
    \|P_N e^{-it \Delta}\phi\|_{L^p_{t, \bx}(I\times \T^d)}
    \les N^{\frac{d}2-\frac{d+2}p} \|P_N \phi\|_{L^2_\bx(\T^d)}.
  \end{align}

\noi
More generally, for all $C \in \mathcal{C}_N$, we have
  \begin{equation}\label{G3}
    \|P_C e^{-it \Delta}\phi\|_{L^p_{t, \bx}(I\times \T^d)}
    \les N^{\frac{d}2-\frac{d+2}p} \|P_C \phi\|_{L^2_\bx (\T^d)}.
  \end{equation}
\end{lemma}

\noi
The Strichartz estimate \eqref{G3}
shows that the loss in \eqref{G3}
depends only on the size of the frequency support, not the position.
See \cite{BoPCMI, HTT11}.

In order to exploit further orthogonality between different frequency pieces
under the linear Schr\"odinger evolution,
we need to decompose the frequency cubes $C_N$.
Let $\mathcal{R}_{M}(N)$ be the collection of all sets in $\Z^d$
which are given as the intersection of a cube of side length $2N$ with a strip of
``width'' $2 M$, i.e.~the collection of all sets of the form
\begin{equation}
  \big(\bn_0+[-N,N]^d\big)\cap \big\{ \bn \in \Z^d: |{\bf a} \cdot_{\pmb{\theta}} \bn -A|\leq M \big\}
\label{G3a}
\end{equation}
with some $\bn_0 \in \Z^d$, ${\bf a} \in \R^d$, $|{\bf a}|=1$, $A \in \R$.
Here,
 the dot product
 ${\bf a} \cdot_{\pmb{\theta}} \bn$ is given by
\begin{equation}
\ba\cdot_{\pmb{\theta}} \bn = \sum_{j=1}^d \theta_j a_j n_j,
\label{G3b}
\end{equation}

\noi
where $\pmb{\theta}  = (\theta_1, \dots, \theta_d)$ is as in \eqref{eq:Q}.
Then, we have the following refinement of the Strichartz estimate.

\begin{lemma}\label{LEM:Gstr2}
Let $(d, p)$ be admissible and  $I\subset \R$ be a bounded interval.
Then,
for all $1\leq M\leq N$ and $R \in \mathcal{R}_M(N)$,  we have
  \begin{equation}\label{G4}
    \|P_R e^{-it \Delta}\phi\|_{L^p_{t, \bx} (I\times \T^d)}
    \les N^{\frac{d}2-\frac{d+2}p} \left(\frac{M}{N}\right)^{\dl}\|P_R \phi\|_{L^2_\bx (\T^d)},
  \end{equation}
\end{lemma}

\noi
where
 $0<\delta <\frac 12-\frac{10}{3p}$ when $ d= 2$,
 $0<\delta <\frac 12-\frac8{3p}$ when $ d= 3$,
 and $0<\dl <\frac 12- \frac 2p$ when $d \geq 4$.

\begin{proof}
By Bernstein's inequality,  we have
\begin{equation}\label{G5}
  \|P_R  e^{-it \Delta}\phi\|_{L^\infty_{t, \bx} (I \times \T^d)}
  \les M^{\frac12}N^{\frac{d-1}{2}} \|P_R \phi\|_{L^2_\bx (\T^d)}.
\end{equation}

\noi
Given admissible $(d, p)$,
write $\frac 1p = \frac {\theta}{q} + \frac{1-\theta}{\infty}$
for some $\theta \in [0, 1)$,
where $q < p$ is given by
$q = \frac{20}{3}+$ when $ d = 2$,
$q = \frac{16}{3}+$ when $ d = 3$,
and $q = 4+$ when $ d\geq 4$.
Then, \eqref{G4} follows
from interpolating \eqref{G3} (with  $p = q$)
and \eqref{G5}.
\end{proof}

Now, we are ready to prove the main  multilinear Strichartz estimates \eqref{G1}.

\begin{proof}[Proof of Proposition \ref{PROP:LWP2}]
Let $u_j = e^{-it \Dl} \phi_j$.
Then, by almost orthogonality in spatial frequencies, it suffices to prove
that there exists $\dl > 0$ such that
  \begin{align}
      \Big\|P_CP_{N_1}u_1\prod_{j=2}^{k+1}P_{N_j} u_j\Big\|_{L^2_{t, \bx}}
&       \les \left(\frac{N_{k+1}}{N_1}+\frac1{N_2}\right)^{\dl}\|P_C P_{N_1}\phi_1\|_{L^2_\bx}
\prod_{j=2}^{k+1}N^{s_c}_j \|P_{N_j}\phi_j\|_{L^2_\bx},
\label{G5a}
  \end{align}

\noi
for all cubes $C \in \mathcal{C}_{N_2}$.
Fix a cube $C \in \mathcal{C}_{N_2}$
and let $\bn_0$ be the center of $C$.
Partition $C = \bigcup R_\l$ into disjoint strips $R_\l$
with width $M=\max(N_2^2/N_1,1)$, which are all orthogonal to $\bn_0$
with respect to the dot product $\cdot_{\pmb{\theta}}$ in \eqref{G3b},
i.e.~$R_\l$ is given by
  \[
   R_\l=\big\{\bn\in C :\,\bn\cdot_{\pmb{\theta}} \bn_0 \in
   \big[|\bn_0| M \l, |\bn_0| M(\l+1)\big) \big\}, \qquad |\l| \sim \frac{N_1}{M}.
  \]

\noi
Note that we have $R_\l \in \mathcal{R}_M(N_2)$.
By writing
  \begin{equation}\label{G6}
   P_C P_{N_1} u_1  \prod_{j=2}^{k+1}P_{N_j} u_j
   = \sum_\l P_{R_\l} P_{N_1} u_1  \prod_{j=2}^{k+1}P_{N_j} u_j,
  \end{equation}

\noi
we show that the sum is almost orthogonal in $L^2_{t, \bx}$.
Since  $N_2^2 \lesssim M^2 \l$,
 we have
  \[
    Q(\bn_1) = \frac1{Q(\bn_0)}|\bn_1 \cdot_{\pmb{\theta}} \bn_0|^2
    + Q(\bn_1-\bn_0) - \frac1{Q(\bn_0)}|(\bn_1- \bn_0) \cdot_{\pmb{\theta}} \bn_0|^2
    = M^2\l^2 + O(M^2 \l),
  \]

\noi
for $\bn_1 \in R_\l$.
Note that the multiplication by the factor $\prod_{j=2}^{k+1}P_{N_j} u_j$
in \eqref{G6}
changes the time frequency  at most by
$O(N_2^2)$.
Hence,
$P_{R_\l} P_{N_1} u_1 \prod_{j=2}^{k+1}P_{N_j} u_j$
in \eqref{G6}
is localized at time frequency $M^2 \l^2 + O(M^2 \l) = O(M^2 \l^2)$
for each $\l$.
Therefore, the sum in \eqref{G6} is almost orthogonal
and we have
  \begin{equation} \label{G6a}
    \Big\|  P_C P_{N_1} u_1  \prod_{j=2}^{k+1}P_{N_j} u_j \Big\|_{L^2_{t, \bx}}^2
    \sim \sum_\l \bigg\|P_{R_\l} P_{N_1} u_1  \prod_{j=2}^{k+1}P_{N_j} u_j\bigg\|_{L^2_{t, \bx}}^2.
  \end{equation}

With $d$ and $k$ as in \eqref{admis1}, let
$p_{d,k}=(d+2)k$.
Then, by  Lemma \ref{LEM:Gstr}, we have
\begin{equation}\label{G7}
  \|P_N e^{-it \Delta}\phi\|_{L^{p_{d,k}}_{t, \bx}}
  \les N^{s_c} \|P_N \phi\|_{L^2_\bx}.
\end{equation}

\noi
Now,
choose $p$ such that
\begin{equation}\label{G8}
  \begin{split}
 \textup{(i)}&  \quad   \frac{20}3<p<\frac{8k}{k+1} \quad  \text{when } d =2, \\
 \textup{(ii)}&  \quad   \frac{16}3<p<\frac{20k}{3k+2} \quad  \text{when } d =3, \\
  \textup{(iii)}& \quad    4<p<\frac{4k(d+2)}{dk+2}\quad  \text{when  } d\ge 4.
  \end{split}
\end{equation}

\noi
The existence of such $p$ is implied by \eqref{admis1}.
Moreover, the lower bound on $p$ guarantees that each $(d,p)$ is admissible,
while the upper bound on $p$ guarantees that
\begin{align}\label{G9}
  d-\frac{2(d+2)}{p}-s_c <0.
\end{align}

\noi
 Let $q$ such that
\begin{align}\label{G10}
  \frac2p+\frac{k-2}{p_{d,k}}+\frac1q=\frac12.
\end{align}

\noi
Then, it follows from \eqref{G8}
that  $(d,q)$ is also admissible.
By H\"older's inequality and Lemmata \ref{LEM:Gstr} and \ref{LEM:Gstr2},
we have
\begin{align}
\bigg\|P_{R_\l} P_{N_1} u_1 & \prod_{j=2}^kP_{N_j} u_j\bigg\|_{L^2_{t, \bx}}
   \les\|P_{R_\l} P_{N_1} u_1 \|_{L^p} \|P_{N_2}u_2\|_{L^p} \prod_{j=3}^k
  \|P_{N_j}u_j\|_{L^{p_{d,k}}} \|P_{N_{k+1}}u_{k+1}\|_{L^q} \notag \\
&    \les N_2^{d-\frac{2(d+2)}{p}-s_c}
    N_{k+1}^{\frac d2- \frac{d+2}q -s_c}\bigg(\frac{M}{N_2} \bigg)^{\dl}
    \|P_{R_\l}P_{N_1}\phi_1\|_{ L^2_\bx}
\prod_{j=2}^{k+1} N^{s_{c}}_j \|P_{N_j}\phi_j\|_{L^2_\bx},
\label{G11}
\end{align}

\noi
for some $\dl > 0$.
In view of \eqref{G8} and \eqref{G9},
choose $p$
such that
\[  - d +\frac{2(d+2)}{p}+s_c = \delta. \]

\noi
Moreover, from \eqref{G9} and \eqref{G10}, we have
\begin{equation}
 \frac d2-\frac{d+2}q-s_c
 =   -d+\frac{2(d+2)}{p}+s_c = \dl> 0.
\label{G12}
\end{equation}

\noi
Then,
 noting $\frac{M}{N_2} \sim \frac{N_2}{N_1}+\frac1{N_2} $,
it follows from \eqref{G11} that
\begin{align*}
   \bigg\|  P_{R_\l} P_{N_1} u_1   \prod_{j=2}^{k+1}P_{N_j} u_j \bigg\|_{L^2_{t, \bx}}
  \les\bigg(
\frac{N_{k+1}}{N_1}+\frac1{N_2}\bigg)^{\delta}
  \|P_{R_\l} P_{N_1}\phi_1\|_{L^2_\bx}  \prod_{j=2}^{k+1}N^{s_c}_j \|P_{N_j}\phi_j\|_{L^2_\bx}.
\end{align*}

\noi
Finally,
by summing up the squares  in \eqref{G6a} with respect to $\l$,
we obtain \eqref{G5a} and hence \eqref{G1}.
This completes the proof of Proposition \ref{PROP:LWP2}.
\end{proof}

\subsection{Proof of Proposition \ref{PROP:XLWP}}
\label{SUBSEC:G3}

%
%
%
%
%

 First, we state and prove an auxiliary lemma (Lemma \ref{LEM:XmultiY}), using Proposition \ref{PROP:LWP2}.
 Let $\mathcal{C}_N$, $N \geq 1$,
 be the collection of cubes $C\subset \Z^d$ of side length $\sim N$ as before.
Let $(d, p)$ be  admissible in the sense of
 Definition \ref{DEF:Gadmis}.
 Then,
 it follows from  Lemma \ref{LEM:Gstr} with Lemma \ref{LEM:Xinterpolate} (i)
 that
   \begin{equation}\label{Y0}
    \|P_C e^{it \Delta}\phi\|_{L^p_{t, \bx} (I\times \T^d)}
    \les N^{\frac{d}2-\frac{d+2}p} \|P_C \phi\|_{U^p_\Dl L^2_\bx}
  \end{equation}

\noi
for all $C \in \mathcal{C}_N$.

\begin{lemma} \label{LEM:XmultiY}
Let $d$ and $k$ satisfy \eqref{admis1}.
Then,  there exists $\dl' > 0$ such that
  \begin{align}
     \bigg\|\prod_{j=1}^{k+1}P_{N_j} u_j\bigg\|_{L^2_{t, \bx}(I \times\T^d)}
     \les
     \bigg(\frac{N_{k+1}}{N_1}+\frac1{N_2}\bigg)^{\dl'}\|P_{N_1}u_1\|_{Y^0}
     \prod_{j=2}^{k+1}\|P_{N_j}u_j\|_{Y^{s_c}},
\label{Y1}
  \end{align}

\noi
for any interval $I \subset [0, 1]$
and
for all $N_1\ge N_2 \ge \cdots \ge N_{k+1}\ge 1$.
\end{lemma}

 \begin{proof}
 By almost orthogonality in spatial frequencies,
 it suffices to prove that there exists $\dl' > 0$ such that
  \begin{align}
     \bigg\|P_CP_{N_1} u_1 \prod_{j=2}^{k+1}P_{N_j} u_j\bigg\|_{L^2_{t, \bx}}
     \les
     \bigg(\frac{N_{k+1}}{N_1}+\frac1{N_2}\bigg)^{\dl'}\|P_C P_{N_1}u_1\|_{Y^0}
     \prod_{j=2}^{k+1} N_j^{s_c}\|P_{N_j}u_j\|_{Y^{0}},
\label{Y2}
  \end{align}

\noi
for all cubes $C \in \mathcal{C}_{N_2}$.
Moreover, by the embedding
\eqref{X3}, it suffices to prove
\eqref{Y2}, where the $Y^0$-norm is replaced by the $V^2_\Dl L^2$-norm.
Furthermore,
it suffices to prove
that there exists $\ld > 0$ such that the following two estimates hold:
  \begin{align}
\text{LHS of \eqref{Y2} }
     & \les
     \bigg(\frac{N_{k+1}}{N_1}+\frac1{N_2}\bigg)^{\dl}\|P_C P_{N_1}u_1\|_{U^2_\Dl L^2}
     \prod_{j=2}^{k+1} N_j^{s_c}\|P_{N_j}u_j\|_{U^2_\Dl L^2},
\label{Y3}
\intertext{and}
\text{LHS of \eqref{Y2} }
     & \les
     \bigg(\frac{N_{k+1}}{N_2}\bigg)^{\dl}\|P_C P_{N_1}u_1\|_{U^p_\Dl L^2}
     \prod_{j=2}^{k+1} N_j^{s_c}\|P_{N_j}u_j\|_{U^p_\Dl L^2},
\label{Y4}
\end{align}

\noi
for some $ p > 2$.
Indeed, if \eqref{Y3} and \eqref{Y4} hold,
then
it follows from Lemma \ref{LEM:Xinterpolate} (ii)
that
 \eqref{Y2} holds with $\dl' < \dl$.

The first estimate \eqref{Y3} directly follows
from Proposition \ref{PROP:LWP2} and
Lemma \ref{LEM:Xinterpolate} (i).
Hence, it remains to prove the second estimate \eqref{Y4}.
Let $p$, $p_{d, k}$, and $q$ be as in the proof of Proposition \ref{PROP:LWP2}.
Then,
by H\"older's inequality with \eqref{Y0},
we have
\begin{align}
\bigg\|P_C P_{N_1} u_1 & \prod_{j=2}^kP_{N_j} u_j\bigg\|_{L^2_{t, \bx}}
   \les\|P_C P_{N_1} u_1 \|_{L^p} \|P_{N_2}u_2\|_{L^p} \prod_{j=3}^k
  \|P_{N_j}u_j\|_{L^{p_{d,k}}} \|P_{N_{k+1}}u_{k+1}\|_{L^q} \notag \\
&    \les N_2^{d-\frac{2(d+2)}{p}-s_c}
    N_{k+1}^{\frac d2- \frac{d+2}q -s_c} \|P_C P_{N_1} u_1 \|_{U^p_\Dl L^2}
\prod_{j=2}^{k+1} N^{s_{c}}_j \|P_{N_j}u_j\|_{U^{q_j}_\Dl L^2},
\label{Y5}
\end{align}

\noi
where $q_2 = p$, $q_j = p_{d, k}$ for $j = 3, \dots, k$, and $q_{k+1} = q$.
From \eqref{G8} and  \eqref{G10} with $p_{d, k} = (d+2) k$,
we have $q > p_{d, k} > p > 2$.
Therefore, \eqref{Y4}
follows from \eqref{Y5}
with  Remark \ref{REM:UpVp} and \eqref{G12}.
 \end{proof}

We conclude this section by presenting the proof of  Proposition \ref{PROP:XLWP}.

\begin{proof}[Proof of Proposition \ref{PROP:XLWP}]

Let $I = [0, T)$.
In the following, we prove
  \begin{align*}
     \Bigg\| \mathcal{I}
  \Bigg( \mathbb{P}_{\le N} \bigg( \prod_{j=1}^{2k+1}u^*_j\bigg) \Bigg)\Bigg\|_{X^{s_c}(I)}
     \les
     \sum_{j = 1}^{2k+1} \bigg(\|u_j\|_{X^s(I)}
     \prod_{\substack{ \l = 1\\ \l \ne j}}^{2k+1} \|u_\l\|_{X^{s_c}(I)}\bigg)
  \end{align*}

\noi
for all $N\geq 1$, where the implicit constant is independent of $N$.
By Lemma \ref{LEM:Xlinear}, we have
  \begin{align}
     \Bigg\| \mathcal{I}
  \Bigg( \mathbb{P}_{\le N} \bigg( \prod_{j=1}^{2k+1}u^*_j\bigg) \Bigg)\Bigg\|_{X^{s_c}(I)}
%
     \le
\sup_{\substack{v \in Y^{-s}([0, T))\\\|v\|_{Y^{-s} = 1}}}
 \bigg| \int_{I \times \T^d} \prod_{j=1}^{2k+1}u^*_j(\bx, t) \mathbb{P}_{\le N} \cj{v(\bx, t)} d\bx dt\bigg|.
\label{Y6}
  \end{align}

\noi
Hence, with $u_0 = \mathbb{P}_{\le N }v$,  it suffices to show that
\begin{align}
 \bigg| \int_{I \times \T^d} \prod_{j=0}^{2k+1}u^*_j(\bx, t)  d\bx dt\bigg|
 \les \|u_0\|_{Y^{-s}(I)}
     \sum_{j = 1}^{2k+1} \bigg(\|u_j\|_{X^s}
     \prod_{\substack{ \l = 1\\ \l \ne j}}^{2k+1} \|u_\l\|_{X^{s_c}}\bigg).
\label{Y7}
  \end{align}

Now, dyadically decompose $u_j^* = \sum_{N_j \geq 1} P_{N_j} u_j^* $.
Without loss of generality, assume
$N_1 \geq N_2 \geq \cdots \geq N_{2k+1}$.
Then, in order to have a non-trivial contribution
on the left-hand side of  \eqref{Y7},
we must have $N_1 \sim \max(N_0, N_2)$.

\medskip

\noi
{\bf Case (i):} $N_0 \sim N_1$.

By Lemma \ref{LEM:XmultiY}, we have
\begin{align}
  \bigg| & \int_{I \times \T^d}  \prod_{j=0}^{2k+1}  P_{N_j} u^*_j(\bx, t)  d\bx dt\bigg|
  \leq
 \bigg\|\prod_{j=0}^{k}P_{N_{2j}} u^*_{2j}\bigg\|_{L^2_{t, \bx}}
 \bigg\|\prod_{j=0}^{k}P_{N_{2j+1}} u^*_{2j+1}\bigg\|_{L^2_{t, \bx}} \notag\\
& \les    \bigg(\frac{N_{2k}}{N_0}+\frac1{N_2}\bigg)^{\dl'}
 \bigg(\frac{N_{2k+1}}{N_1}+\frac1{N_3}\bigg)^{\dl'}
\| P_{N_0} u_0\|_{Y^{-s}} \| P_{N_1} u_1\|_{Y^{s}}    \prod_{j=2}^{2k+1}\|P_{N_j}u_j\|_{Y^{s_c}}.
 \label{Y8}
  \end{align}

\noi
Summing \eqref{Y8} over dyadic blocks $N_0\sim N_1\geq   N_2 \geq \cdots \geq N_{2k+1}$
and by Cauchy-Schwarz inequality,
we have
\begin{align*}
\text{LHS of \eqref{Y7} }
& \les \sum_{N_0 \sim N_1} \| P_{N_0} u_0\|_{Y^{-s}} \| P_{N_1} u_1\|_{Y^{s}}
\prod_{j=2}^{2k+1}\|u_j\|_{Y^{s_c}} \notag \\
& \les \|  u_0\|_{Y^{-s}} \|  u_1\|_{Y^{s}}
\prod_{j=2}^{2k+1}\|u_j\|_{Y^{s_c}},
  \end{align*}

\noi
yielding  \eqref{Y7} in view of \eqref{X3}.

\medskip

\noi
{\bf Case (ii):} $N_2 \sim N_1 \gg N_0$.

By Lemma \ref{LEM:XmultiY} with $N_1\sim N_2$, we have
\begin{align}
  \bigg|  \int_{I \times \T^d} &  \prod_{j=0}^{2k+1}  P_{N_j} u^*_j(\bx, t)  d\bx dt\bigg|
\notag\\
& \les
\bigg(\frac{N_0}{N_1}\bigg)^{s+s_c}
 \bigg(\frac{N_{2k+1}}{N_1}+\frac1{N_3}\bigg)^{\dl'}
\| P_{N_0} u_0\|_{Y^{0}} \| P_{N_1} u_1\|_{Y^{s}}    \prod_{j=2}^{2k+1}\|P_{N_j}u_j\|_{Y^{s_c}}
 \label{Y9}
  \end{align}

\noi
Summing \eqref{Y9} over dyadic blocks
$N_0 (\ll N_1)$ and $ N_1 \sim N_2 \geq N_3 \geq \cdots \geq N_{2k+1}$
and by Cauchy-Schwarz inequality,
we have
\begin{align*}
\text{LHS of \eqref{Y7} }
& \les
\| u_0\|_{Y^{-s}}
\sum_{N_1 \sim N_2}  \| P_{N_1} u_1\|_{Y^{s}}     \| P_{N_2} u_2\|_{Y^{s_c}}
\prod_{j=3}^{2k+1}\|u_j\|_{Y^{s_c}} \notag \\
& \les \|  u_0\|_{Y^{-s}} \|  u_1\|_{Y^{s}}
\prod_{j=2}^{2k+1}\|u_j\|_{Y^{s_c}}.
  \end{align*}

\noi
This completes the proof of Proposition \ref{PROP:XLWP}.
\end{proof}

\appendix

\section{On the Weyl sum estimate \eqref{A6}}\label{SEC:A}
%
%
%
%

In this appendix, we present a proof of \eqref{A6}.
We decided to include the proof for the convenience of readers, in particular for those in PDEs.
Since $I = I(\eta)$ is a compact interval and the integrand is periodic with period 1,
it suffices to show
\begin{equation}
\int_0^{1} | F(t) |^{r}dt \sim N^{r-2},
\label{Z1}
\end{equation}

\noi
for $ r > 4$,
where $F(t)$ is the Weyl sum defined by
\begin{align*}
F(t) =  \sum_{0\le n\le N} e^{2\pi i n^2 t}.
\end{align*}

For  $a, q \in \mathbb{N}$
with $1 \leq a \leq q \leq N$ and $(a, q) = 1$,
define a major arc $\mathfrak{M}(q, a)$ by
\begin{equation}
\mathfrak{M}(q, a)
= \bigg\{ t \in [0, 1]: \, \Big\|t - \frac{a}{q}\Big\| \leq \frac{1}{100N^2}\bigg\},
\label{Z1a}
\end{equation}

\noi
where $\|x\| = \min_{n\in \Z}|x - n|$ denotes the distance
of $x$ to the closest integer as before.
Let $\mathfrak{M} = \bigcup_{a, q} \mathfrak{M}(q, a)$.
Note that  we have
$\big\|t - \frac{a}{q}\big\| \leq \frac{1}{q^2}$ for $t \in \mathfrak{M}$.
Then, by
Weyl's inequality 
we have
\begin{equation}
 |F(t) | \les
\frac{N}{q^\frac{1}{2}} + N^\frac{1}{2}(\log q)^\frac{1}{2} + q^\frac{1}{2}(\log q)^\frac{1}{2}.
\label{Z2}
\end{equation}

\noi
Hence, the contribution from the major arc $\mathfrak{M}$ is estimated by
\begin{align*}
\int_\mathfrak{M} | F(t) |^{r} dt
\les \sum_{q = 1}^N \sum_{\substack{a=1\\(a, q) = 1}}^q\frac{N^r}{q^\frac{r}{2}}(\log q)^\frac{r}{2} \frac{1}{N^2}
\leq  N^{r-2} \sum_{q = 1}^N q^{1 - \frac r2+}
\les  N^{r-2},
\end{align*}

\noi
since $r > 4$.

\begin{remark} \rm
Indeed,  the contribution from the major arc $\mathfrak{M}$
provides the lower bound in \eqref{Z1}.
We only need to consider the contribution from $\mathfrak{M}(q, a)$
for odd $q \leq N^\frac{1}{2}$.
Let $S(q, a)$ be the Gauss sum
given by
$S(q, a) = \sum_{n = 1}^q e^{2\pi i n^2 \frac aq}$.
We have $S(q, a) = \sqrt{q}$ for odd $q$.
Now, suppose that $q$ is odd such that $1\leq q \leq N^\frac{1}{2}$.
Then, by Van der Corput's method \cite{Vi} with $S(q, a) = \sqrt{q}$,
one can show that $|F(t)|  \gtrsim \frac{N}{q^\frac{1}{2}}$
for $t \in \mathfrak{M}(q, a)$.
Noting that $\mathfrak{M}(q, a)$ are disjoint, we have
\begin{align*}
\int_\mathfrak{M} | F(t) |^{r} dt
\gtrsim \sum_{\substack{q = 1\\q,  \text{odd}}}^{N^\frac{1}{2}}
 \sum_{\substack{a=1\\(a, q) = 1}}^q\frac{N^r}{q^\frac{r}{2}}
 \frac{1}{N^2}
=  N^{r-2} \sum_{\substack{q = 1\\q,  \text{odd}}}^{N^\frac{1}{2}} \phi(q) q^{ - \frac r2}
\sim  N^{r-2},
\end{align*}

\noi
where
$\phi(q)$ denotes Euler's function.
This shows the lower bound in \eqref{Z1}.
\end{remark}

Next, we estimate the contribution from the minor
arc $\mathfrak{m} : = [0, 1] \setminus \mathfrak{M}$.
Fix small $\eps > 0$.
Then, Dirichlet's theorem \cite[Lemma 2.1]{V}
states that given $t \in [0, 1]$,
there exist integers $a, q$
with $1 \leq a \leq q \leq  N^{2-\eps}$ and  $(a, q) = 1$
such that $\big\|t - \frac{a}{q}\big\| \leq \frac{1}{qN^{2-\eps}}$.
Define $I(q)$ by
\[ I(q) = \bigcup_{\substack{a = 1\\(a, q) = 1}}^q I(q, a), \]

\noi
where
  $I(q, a) = \big\{ t\in [0, 1]:\,  \big\|t - \frac{a}{q}\big\| \leq \frac{1}{qN^{2-\eps}} \big\}$.
Now, in view of \eqref{Z1a},
divide the minor arc  $\mathfrak{m}$ into two pieces:
$ \mathfrak{m} = \mathfrak{m}_1 \cup \mathfrak{m}_2$,
where
\[ \mathfrak{m}_1 = \mathfrak{m}\cap
\bigcup_{ 1\leq  q \ll N^\eps} I(q)
\quad \text{and}
\quad
\mathfrak{m}_2 = \mathfrak{m}\cap \bigcup_{ N < q \leq  N^{2-\eps}} I(q).\]


Let $ t \in \mathfrak{m}_2$.
From \eqref{Z2}, we have  $|F(t)|\les N^{1-\frac{1}{2}\eps} (\log N)^\frac{1}{2}$.
Then, by Hua's inequality \cite[Lemma 2.5]{V}, we have
\begin{align*}
\int_{\mathfrak{m}_2} | F(t) |^{r} dt
& \leq \Big(\sup_{t \in \mathfrak{m}_2}|F(t)|\Big)^{r-4}
\int_0^1 | F(t) |^4 dt
\les \big[ N^{1-\frac{1}{2}\eps} (\log N)^\frac{1}{2}\big]^{r-4}
N^{2+} \\
& \leq N^{r-2}.
\end{align*}

Let $t \in \mathfrak{m}_1$, i.e.~we have  $\big\|t - \frac{a}{q}\big\| \leq \frac{1}{qN^{2-\eps}}$
for some $ q \ll N^\eps$.
Then, by Lemmata 2.7 and 2.8 in \cite{V} with $|S(q, a)| \les q^\frac{1}{2}$, we have
\begin{equation*}
|F(t)| =   q^{-1} S(q, a) v\big( t- \tfrac aq\big)+ O(N^\frac{2}{200})
\les N^{1-\frac{\eps}{2}},
\end{equation*}

\noi
where $v$ is defined in \cite[(2.9)]{V}.
Applying Hua's inequality as before, we have
\begin{align*}
\int_{\mathfrak{m}_1} | F(t) |^{r} dt
& \leq \Big(\sup_{t \in \mathfrak{m}_1}|F(t)|\Big)^{r-4}
\int_0^1 | F(t) |^4 dt
\les \big[N^{1-\frac{\eps}{2}} \big]^{r-4}
N^{2+}
 \leq N^{r-2}.
\end{align*}

\noi
This completes the proof of \eqref{Z1}.

\section{On a partially irrational torus $\T^2 \times \T_{\al_3}$}
\label{SEC:B}

In this appendix, we present a sketch of the proof of Theorem \ref{THM:5}.
Thus, we set $d = 3$ and assume that \eqref{eq:Q1} holds in the following.
The main ingredient is an improvement of the Strichartz estimate
on a partially irrational torus.

\begin{proposition}\label{PROP:Bstrichartz}
Let $d = 3$ and $I$ be a compact interval in $\R$.
Suppose that \eqref{eq:Q1} holds.
Then, the  scaling-invariant Strichartz estimate \eqref{P1} holds
for $p > \frac{14}3$.
\end{proposition}

\noi
The proof of Proposition \ref{PROP:Bstrichartz}
is based on the following level set estimates
under the assumption \eqref{eq:Q1}.

\begin{lemma}\label{LEM:level2}
Suppose that \eqref{eq:Q1} holds.
Given  a compact interval $I \subset \R$
and  $f$ as in \eqref{C0},
let $A_\ld = A_\ld(f)$ be the distribution function defined by \eqref{C00a}.

\noi
\textup{(i)}
For any $\eps > 0$, we have
\begin{equation}
|A_\ld |
\les N^{\frac{1}{1+4\eps} } \ld^{-4 + \frac{8}{1+4\eps}\eps}
\label{BB1}
\end{equation}

\noi
for $\ld \ges N^{1+\eps}$.

\medskip
\noi
\textup{(ii)}
Let  $ q > 4$.
Then, there exists small $\eps >0$ such that
\begin{equation}
|A_\ld |\les N^{\frac{3}{2} q - 5} \ld^{-q}
\label{BB2}
\end{equation}

\noi
for $\ld \ges N^{\frac{3}{2}-\eps}$.

In \eqref{BB1} and \eqref{BB2},
 the implicit constants depend on $\eps > 0$, $q >4$, and $|I|$,
but are independent of $f$.

\end{lemma}

We first  prove  Proposition \ref{PROP:Bstrichartz}
assuming Lemma \ref{LEM:level2}.
Then, we sketch the proof of Theorem \ref{THM:5}.
We present the proof of  Lemma \ref{LEM:level2}
at the end of this appendix.

Given $f$ as in \eqref{C0},
let $F(\bx, t) = e^{-it \Dl} f(\bx, t)$.
By Cauchy-Schwarz inequality, we have
$\|F\|_{L^\infty_{t, \bx}} \les N^\frac{3}{2}. $
Given $p > \frac{14}{3}$, let $q \in (4, p)$.
Then,
by Lemma \ref{LEM:level2} and Theorem \ref{THM:1} (ii),
we have
\begin{align*}
& \int_{I \times \T^3 }  |F(\bx, t)|^p d\bx dt\\
& \hphantom{X}
\leq \bigg(\int_{ N^{1+\eps_1}\les |F| \les N^{\frac{3}{2}-\eps_2}}
+ \int_{ N^{\frac 32 - \eps_2}\les |F| \les N^{\frac{3}{2}}}\bigg)
|F(\bx, t)|^p d\bx dt
+ N^{(1+\eps_1)(p - 4)}\int |F(\bx, t)|^4 d\bx dt\\
& \hphantom{X}
 \les N^{1-\frac{4}{1+4\eps_1}\eps_1}
\int_{N^{1+\eps_1}}^{N^{\frac{3}{2}-\eps_2}}\ld^{p - 4+ \frac{8}{1+4\eps_1}\eps_1} d\ld
+  N^{\frac{3}{2}q - 5}
\int_{N^{\frac{3}{2}-\eps_2}}^{N^\frac{3}{2}}\ld^{p - 1 - q} d\ld
 + N^{(1+\eps_1)(p - 4)+ \frac 43+}\\
& \hphantom{X}
 \les N^{\frac{3}{2}p - 5},
\end{align*}

\noi
for sufficiently small $\eps_1, \eps_2 > 0$.
Here, the condition $p > \frac{14}{3}$ is needed in  the last inequality.
This proves
Proposition \ref{PROP:Bstrichartz}.


Next, we briefly discuss how  Theorem \ref{THM:5} follows from
 Proposition \ref{PROP:Bstrichartz}.
As in Section \ref{SEC:critical},
the goal is to prove the multilinear estimate \eqref{X0}
in Proposition \ref{PROP:XLWP}
for $ d = 3$ and $k=2$
(with $s_c =1$) under the assumption \eqref{eq:Q1}.
In view of the argument in Subsection \ref{SUBSEC:G3}
(which holds without any change even in this case),
it suffices to prove
the multilinear Strichartz estimate \eqref{G1} in
Proposition \ref{PROP:LWP2}
for $d = 3$ and $k = 2$ (with $s_c = 1$).
By repeating the proof of Lemma \ref{LEM:Gstr2}
with  Proposition \ref{PROP:Bstrichartz},
we see that \eqref{G4} holds
for $ p> \frac{14}{3}$
with $\dl \in (0,  \frac 12 -\frac7{3p})$.
In the proof of Proposition \ref{PROP:LWP2},
the only change appears in \eqref{G8}
and  we can choose $p$ such that
\[ \frac{14}{3} < p< \frac{20k}{3k+2}\]

\noi
in this case.
In particular, we can set $k = 2$
by choosing $p \in (\frac{14}{3}, 5)$.
The rest of the argument follows exactly as in Section \ref{SEC:critical}.

In the remaining part of the paper, we present the proof of Lemma \ref{LEM:level2}.

\begin{proof}[Proof of Lemma \ref{LEM:level2}]
The proof follows the proof of Proposition \ref{PROP:level}
with small modifications.
In the following, we only point out these modifications.
Given an interval $I \subset \R$,
assume that $I $ is centered at $0$.

\medskip

\noi
(i)
With $Q(\bn)$ in \eqref{eq:Q1},
define $\mathbf{R}$ as in \eqref{B4a}.
Then, we have
$\mathbf{R}(\bx, t) = \prod_{j = 1}^2 R_1(x_j, t)\cdot
R_{\theta_3}(x_3, t),$
where $R_\theta$ is defined in \eqref{B4b}.
We also define
$\mathbf{R}_1$ and $\mathbf{R}_2$ as in \eqref{C2a}.
Then, by \eqref{C2b}, we have
\begin{equation}
\|\RR_1\|_{L^\infty_{t, \bx}}
\les \min\big( N M\log M, N^3\big).
\label{BB3}
\end{equation}

\noi
Proceeding as in \eqref{C6} with \eqref{BB3} and \eqref{C5},
we have
\begin{align}
\ld^2 |A_\ld|^2
&  \leq \|\RR_1\|_{L^\infty_{t, \bx}} \|\Theta_\ld\|_{L^1_{t, \bx}}^2
+ \|\ft \RR_2\|_{L^\infty_\tau \l^\infty_\bn } \|\Theta_\ld\|_{L^2_{t, \bx}}^2  \notag \\
& \leq C_1 N  M^{1+\eps_1} |A_\ld|^2
+ M^{-1+\eps_2} |A_\ld|
\label{BB4}
\end{align}

\noi
for small $\eps_1, \eps_2 > 0$.
Choose $M \geq N$ such that
$N M^{1+\eps_1} \sim \ld^2.$
This condition  with $M \geq N$
implies that $\ld \ges N^{1 + \frac{\eps_1}{2}}$.
Then,
\eqref{BB4} yields
\begin{align*}
|A_\ld| \les \Big(\frac{N}{\ld^2}\Big)^\frac{1 - \eps_2}{1+\eps_1} \ld^{-2}
\les N^{\frac{1}{1+2\eps_1} } \ld^{-4+\frac{4}{1+2\eps_1}\eps_1} 
\end{align*}

\noi
by setting $\eps_2 = \eps_2(\eps_1)$ such that
$\frac{1 - \eps_2}{1+\eps_1} = \frac{1}{1+2\eps_1}$.
This proves \eqref{BB1} with $\eps =  \frac{\eps_1}{2}$.

\medskip

\noi
(ii)
Define $\K$ by
\begin{align}
\K(\bx, t) 
& = \chi(t) K(x_1, t) K(x_2, t)  \sum_{|n_3|\leq N} e^{2\pi i (n_3 x_3 + \theta_3 n_3^2 t)} ,
\label{BB5}
\end{align}

\noi
where  $K(x, t)$ is as in \eqref{D0a}.
Let $\Ld_{M, s}$ be as in \eqref{D7}.
Then, from Lemma \ref{LEM:Weyl}, \eqref{D4}, and \eqref{D5} with $M \leq 2^s \leq N$, we have
\begin{align}
|\Ld_{M, s}(\bx, t)|\les N \bigg(\frac{N}{M^\frac{1}{2} \big( 1 + (2^{-s} N)^\frac{1}{2}\big)}\bigg)^2
+ \frac{M^2}{2^sN} N^{2}
\les \frac{2^s N^2 }{M}.
\label{BB6}
\end{align}

\noi
Hence, from \eqref{BB6}, we have
\begin{equation}
\| f * \Ld_{M, s} \|_{L^\infty(I \times \T^d)}
\les \frac{2^s N^2 }{M}\|f\|_{L^1(I \times \T^d)}.
\label{BB7}
\end{equation}

\noi
Also, with $\Ld$ as in \eqref{E1} and $\K$ as in \eqref{BB5},
it follows from
 \eqref{D4} and  \eqref{D5} with $N_1 = \frac{1}{100}N$ that
\begin{equation}
\|\K - \Ld\|_{L^\infty(I\times\T^d)}
\les N^{2}.
\label{BB8}
\end{equation}

Let $p$ be as in \eqref{E4a}. 
In the following, we consider the case $\theta \geq \frac{1}{3}$.
With $M_1$ and  $M_2$ as in \eqref{EE0},
write $\Ld = \Ld_1 + \Ld_2+ \Ld_3$ as before.
See \eqref{EE0a}.

By interpolating \eqref{BB7} and \eqref{D13a}
and summing over dyadic $M\leq M_1$ and $2^s \leq N$, we have
\begin{align}
\| f *\Ld_1  \|_{L^{p'}(I \times \T^d)}
& \les M_1^{-1 + (3 +)\theta} N^{3 - 5\theta}
\| f \|_{L^{p}(I \times \T^d)},
\label{BB9}
\end{align}

\noi
since $\theta \in [\frac 13, \frac 12)$.

Next,
choose $D \sim M^\al$ for some small $\al = \al(\theta) > 0$,
and set $\beta \ll 1$ and $B \gg1$ such that
\[ \s : = -3 -\be +\al B - > 0.\]

\noi
Then,
by interpolating \eqref{BB7} and \eqref{D18}
and summing over dyadic $M \in ( M_1, M_2]$ and $2^s \leq N$, we have
\begin{align}
\| f * \Ld_2   \|_{L^{p'}(I \times \T^d)}
&  \les
N^{3 - 5 \theta}
 \| f \|_{L^{p}(I \times \T^d)} \notag \\
& + \Big( M_1^{-1 - \s \theta} N^{3 - \frac 52\theta}
+
M_2^{-1 + (3+  B+) \theta} N^{3 - \frac 72 \theta} \Big)
\|f\|_{L^1(I\times \T^d)},
\label{BB10}
\end{align}

\noi
as long as
$\theta < \frac{1}{2+\al +}.$
This can be guaranteed by choosing $\al = \al(\theta) > 0$ sufficiently small
for given $\theta < \frac{1}{2}$.

Lastly,
by interpolating \eqref{BB7} and \eqref{D13}
and  summing over $M \geq M_2 \sim N^{\dl_2}$ and $s$ with $2^s \leq N$,
we have
\begin{align}
\| f * \Ld_3   \|_{L^{p'}(I \times \T^d)}
& \les N^{3 - 5\theta}  
\| f \|_{L^{p}(I \times \T^d)},
\label{BB11}
\end{align}

\noi
as long as $\theta < \frac 12$.

Now, we are ready to prove the level set estimate \eqref{BB2} for $ q > 4$.
Then, proceeding as before
with \eqref{BB8}, \eqref{BB9}, \eqref{BB10}, and \eqref{BB11},
we have
\begin{align}
\ld^2|A_\ld|^2
& \leq
|\jb{ (\K - \Ld) * \Theta_\ld, \Theta_\ld  }|+ |\jb{ \Ld*\Theta_\ld, \Theta_\ld }|\notag \\
& \les N^{2} |A_\ld|^2
+  M_1^{-1 + (3+) \theta} N^{3 - 5 \theta} |A_\ld|^\frac{2}{p} \notag \\
& \hphantom{XXX}
+ M_1^{- 1- \s \theta} N^{3 - \frac 52 \theta}|A_\ld|^{1+\frac{1}{p}}
+ M_2^{ -1+(3+B+) \theta} N^{3 - \frac 72 \theta}
|A_\ld|^{1+\frac{1}{p}}.
\label{BB12}
\end{align}

\noi
Since $\ld \gg N$,
\eqref{BB12} reduces to
\begin{align}
\ld^2|A_\ld|^2
& \les
  M_1^{-1 + (3 +)\theta} N^{3 - 5 \theta} |A_\ld|^\frac{2}{p}
+ M_1^{- 1- \s \theta} N^{3 - \frac{5}{2} \theta}|A_\ld|^{1+\frac{1}{p}}\notag \\
& \hphantom{XXX}
+ M_2^{ -1+(3+B+) \theta} N^{ 3- \frac 7 2 \theta}
|A_\ld|^{1+\frac{1}{p}} \notag \\
%
& = : \I + \II + \III.
\label{BB13}
\end{align}

First, suppose that $\ld^2|A_\ld|^2 \les \I$ holds.
With  $p' \theta = 2$ and \eqref{EE0},
we have
\begin{align}
|A_\ld| \les
\bigg(\frac{N^\frac{3}{2}}{\ld}\bigg)^{(-\frac{p'}{2}+3+)\dl_1}
N^{\frac{3}{2}p' - 5 } \ld^{-p'}
\les
N^{\frac{3}{2}q - 5 } \ld^{-q}
\label{BB14}
\end{align}

\noi
for $ q > p'$ by choosing $\dl_1 = \dl_1(q, p')$ sufficiently small.
Next, suppose that
$\ld^2|A_\ld|^2 \les \II$ holds.
Then,  from  \eqref{EE0},
we have
\begin{align}
|A_\ld|
& \les
N^{\frac{3}{2}p' - 5 } \ld^{-p'}
\Big(N^{\frac{3}{2}p'} \ld^{-p'}M_1^{-p' - 2\s}\Big)
 \les N^{\frac 3 2 p' - 5} \ld^{-p'},
\label{BB15}
\end{align}

\noi
by making $\s = \s(p', \dl_1)$ in \eqref{E5a}
(and hence $B = B(p', \dl_1)$)
sufficiently large.
Lastly, suppose that
$\ld^2|A_\ld|^2 \les \III$ holds.
By $\ld \ges N^{\frac{3}{2}-\eps}$ and  \eqref{EE0},
 we have
\begin{align}
|A_\ld|
& \les
 N^{\frac 3 2 p' - 5} \ld^{-p'}
N^{-2 + \eps p' +(-p' + 6+ 2B+)\dl_2}
\les  N^{\frac 3 2 p' - 5} \ld^{-p'}
\label{BB16}
\end{align}

\noi
as long as
we have $\eps p' \leq 1$
and
$\dl_2 = \dl_2(p', B)$ is sufficiently small.

Finally, given $q > 4$,
we choose $\theta < \frac{1}{2}$ such that
$ q > p' = \frac{2}{\theta} > 4$.
Then,  \eqref{BB2} follows from \eqref{BB13}, \eqref{BB14}, \eqref{BB15}, and \eqref{BB16}
with $\ld \leq N^\frac{3}{2}$.
%
%
%
%
\end{proof}

%

\end{document}